\documentclass[11pt, leqno]{amsart}
\usepackage{amssymb}
\usepackage{amsmath,amscd}
\usepackage[initials,nobysame,alphabetic]{amsrefs}
\usepackage{amsmath, amssymb}
\usepackage{amsfonts}
\usepackage{mathrsfs}
\usepackage{mathpazo}
\usepackage{mathtools}
\usepackage[arrow,matrix,curve,cmtip,ps]{xy}
\usepackage[utf8x]{inputenc}
\usepackage{amsthm}

\usepackage{float}
\usepackage{hyperref}
\usepackage{graphics}
\usepackage{graphicx}
\usepackage{verbatim}
\usepackage{multirow}
\usepackage{tikz}
\usepackage{enumerate}
\allowdisplaybreaks

\newtheorem{theorem}{Theorem}[section]
\newtheorem{lemma}[theorem]{Lemma}
\newtheorem{proposition}[theorem]{Proposition}
\newtheorem{corollary}[theorem]{Corollary}

\newtheorem*{theorem*}{Theorem}
\newtheorem{remark}[theorem]{Remark}

\def\as{\hbox{\rm a.s.{ }}}

\numberwithin{equation}{section}

\usepackage{color}

\newtheorem{Definition}{Definition}[part]

\newtheorem{Remark}{Remark}[part]
\newtheorem{Theorem}{Theorem}[part]

\def \unb{\underbar}

\def \t{\tau}
\def \ind{1\!\!1}

\newcommand{\nc}{\newcommand}
\nc{\esssup}{\mathop{\mathrm{ess\,sup}}}
\nc{\essinf}{\mathop{\mathrm{ess\,inf}}}
\nc{\argmax}{\mathop{\mathrm{arg\,max}}}
\def \lb {\label}
\def \nd {\noindent}
\def \ms {\medskip}
\def \bs {\bigskip}

\def \P{\mathbb{P}}

\def \R{\mathbb{R}}

\def \E{\mathbb{E}}

\def \1{\mathbf{1}}

\def \Dc{\mathcal{D}}

\def \Fc{\mathcal{F}}

\def \Hc{\mathcal{H}}

\def \Pc{\mathcal{P}}
\def \Mc{\mathcal{M}}

\def \Sc{\mathcal{S}}
\def \Tc{\mathcal{T}}

\def \eps{\varepsilon}
\def \nn {\nonumber}
\def \ed {\end{document}}

\def \hy {\hat Y}
\def \hz {\hat Z}
\def \hk {\hat K}

\def\reff#1{{\rm(\ref{#1})}}

\def\enqs{\end{eqnarray*}}
\def\beq{\begin{eqnarray}}
\def\enq{\end{eqnarray}}

\def \und {\underline}
\def \sp {\mathcal{S}_{c}^{p}}
\def \ga {\gamma_1}
\def \gb {\gamma_2}
\def \d {\delta}
\begin{document}
\title[Mean-Field Reflected BSDEs]{Mean-field reflected backward stochastic differential equations}

\author{Boualem Djehiche, Romuald Elie and  Said Hamad{\`e}ne}

\address{Department of Mathematics \\ KTH Royal Institute of Technology \\ 100 44, Stockholm \\ Sweden}
\email{boualem@kth.se}
\address{LAMA UMR CNRS 8050, Université Paris-Est, Champs-sur-Marne, France}
\email{romuald.elie@univ-mlv.fr} 
\address{LMM, Le Mans University, Avenue Olivier Messiaen, 72085 Le Mans, Cedex 9, France}
\email{hamadene@univ-lemans.fr}
\thanks{{\bf Acknowledgements}. The first author gratefully acknowledges the financial support provided by the Swedish Research Council grant (2016-04086)}

\date{\today}



\begin{abstract}
In this paper, we study a class of reflected backward stochastic differential equations (BSDEs) of mean-field type, where the mean-field interaction in terms of the distribution of the $Y$-component of the solution enters in both the driver and the lower obstacle. We consider in details the case where the lower obstacle is a deterministic function of $(Y,\E[Y])$ and discuss the more general dependence on the distribution of $Y$. Under mild Lipschitz and integrability conditions on the coefficients, we obtain the well-posedness  of such a class of equations.  Under further monotonicity conditions, we show convergence of the  standard penalization scheme to the  solution of the equation, which hence satisfies a minimality property. This class of equations is motivated by applications in pricing life insurance contracts with surrender options.  
\end{abstract}

\maketitle



\section{Introduction}

\medskip
Backward stochastic differential equations (BSDEs) have been extensively studied in a variety of context since the seminal paper by Pardoux and Peng \cite{parpeng90}. Much of the interest in BSDEs is due to the induced probabilistic representation of solutions to a large class of semilinear PDEs and stochastic control problems. Hereby, it constitutes a powerful tool for investigating several meaningful applications in engineering, investment science including mathematical finance, game theory and insurance, among many other areas.

Given a square integrable terminal condition $\xi$ and a Lipschitz continuous driver $f$, a solution to a typical BSDE consists of a pair of adapted processes $(Y,Z)$ defined on a filtered probability space
$(\Omega,\mathcal{F},(\mathcal{F}_t)_t,\P)$, on which is defined a Brownian motion $(B_t)_t$, which satisfy
$$
Y_t=\xi+\int_t^T f(s,Y_s,Z_s)ds-\int_t^T Z_sdB_s, \quad 0\le t\le T,
$$
where $Z$ is a control process which ensures that $Y$ may be expressed as a recursive utility function
$$
Y_t=E\left[\xi+\int_t^T f(s,Y_s,Z_s)ds\,|\,\mathcal{F}_t\right].
$$
In investment science (see e.g.\@ \cite{duffie1}, \cite{duffie2}, \cite{elkaroui2}), this formulation has the plausible interpretation of $Y$ as the yield of an investment under uncertainty, and $Z$ characterizes the optimal investment strategy. In Norberg \cite{norberg1}, \cite{norberg2}, $Y$ is interpreted as the prospective reserve of a life insurance contract for which the driver $f$ represents the payment process of contractual benefits less premiums payable during the time interval $dt$ and $\xi$ identifies to the terminal payment at the horizon $T$.

Given some constraint such as a minimum required amount $L_t$, a.k.a. solvency constraint, at each time $t$, one would like to find an optimal time $D_t$ after $t$ at which one can exit the investment so that  at any date,  the yield $Y$ is always kept larger than the constraint $L$:
\begin{equation}\label{reflec-0}
Y_t\ge L_t, \quad 0\le t\le T.
\end{equation}
This amounts to express the yield $Y$ as a solution of an optimal stopping problem:
\begin{equation*}
Y_t=\underset{\tau\,\,\mbox{stopping time\,\,}\ge t}{\esssup\,}\E[\int_t^\tau f(s,Y_s,Z_s)ds+L_{\tau}1_{\{\tau<T\}}+\xi1_{\{\tau=T\}}|F_t],\quad t\le T.
\end{equation*}
Intuitively, the smallest optimal time $D_t$ should be the first instant $s$ after $t$ where $Y_s$ reaches the constraint $L_s$:
$$
D_t=\inf\{s\ge t,\,\, Y_s=L_s\}\wedge T.
$$
El Karoui {\it et al.} \cite{elkaroui1} were the first to show that such a pair $(Y,Z)$ of adapted processes satisfies
\begin{equation}\label{rbsde-0}
Y_t = \xi + \int_t^T f(s,Y_s,Z_s)ds - \int_t^T Z_sdB_s + K_T-K_t, \quad 0\le t\le T,
\end{equation}
where the extra term $K$ is an adapted increasing process for which $K_T-K$ is the running cost for keeping $Y$ above $L$ at all times. In connection to optimal stopping problems, this cost must be  minimal in the sense that $K$ is only required to push the investment yield $Y$ above $L$, whenever it may cross it. Namely, whenever $Y_t>L_t$, there is no reason to stop at time $t$ so that $dK_t=0$. Hereby, the minimal solution of interest to \reff{rbsde-0} is uniquely characterized  by the famous Skorohod flatness condition for Snell enveloppes 
\begin{equation}\label{skoro-0}
\int_0^T (Y_t-L_t)^+dK_t=0.
\end{equation}
The dynamics \eqref{rbsde-0}-\eqref{reflec-0}-\eqref{skoro-0} is called reflected BSDE whose solution is the adapted process $(Y,Z,K)$. This class of constrained BSDEs has been widely extended in different directions, in relation to their possible connections to zero-sum games \cite{CviKar}, investment strategies with portfolio constraints \cite{CviKarSon}, switching problems \cite{HamJea,EliKha}, robust optimal stopping \cite{matoussi2013} and many others. We refer to the recent paper \cite{briandetal} for further discussions and references.

Motivated by considerations related to partial hedging of financial derivatives in mathematical finance,  Briand {\it et al.} \cite{briandetal} built on the weak hedging approach considered in \cite{BouEliRev} and introduce  a class of BSDEs where the pathwise running constraint $Y_t\ge L_t$ is replaced by a weaker constraint on the distribution of the yield $Y$ of the form
\begin{equation}\label{intro-0}
\E[\ell(Y_t-L_t)]\ge 0\;, \quad \quad 0\le t\le T,
\end{equation}
for a given loss function $\ell$. A typical situation is when the process $Y$ is required to remain above a certain benchmark (or solvency level) $u$ with a probability higher than some given level $v$ in which case $\ell:(t,x)\mapsto\ind_{\{x\ge u_t\}}-v_t$. For that reason, it is known in the literature as BSDE with mean reflection. 
Under Lipschitz and integrability conditions on $f$ and $\xi$, as well as regularity  and monotonicity of $\ell$, they obtain a unique solution $(Y,Z,K)$ with {\it deterministic} $K$ to the BSDE with mean reflection with the Skrohod type condition
\begin{equation}\label{skoro-0}
\int_0^T \E[\ell(Y_t-L_t)]^+dK_t=0.
\end{equation}
Observe that the constraint \reff{intro-0} is much weaker than the one considered in \cite{Dumetal}, where the constraint in expectation \eqref{intro-0} must be satisfied for any stopping time, although such distinction is not relevant for strong reflections of the form \reff{reflec-0}.\\
In life insurance and pension (see the example in Section 2.1 below), the lower barrier $L$, usually interpreted as a dynamical solvency level, is typically of the form
$$
L_t=u-c(Y_t)+\lambda(\E[Y_t]-u)^+\;, \quad \quad 0\le t\le T,
$$
where $u$ is a required minimum or a benchmark return, $c(Y_t)$ is a reserve dependent management fee and  $\lambda(\E[Y_t]-u)^+$ is a bonus option i.e. a fraction $\lambda$ of the possible surplus realized by the mean value of $Y_t$. This extra term typically reflects the cooperative aspects induced by the pooling principles of insurance policies.\\

Motivated by this example, the purpose of this paper is to study of
the following fairly general class of reflected BSDEs of mean-field
type
$$\left\{
\begin{array}{lll}
Y_t=\xi+\int_t^T f(s,Y_s,\P_{Y_s},Z_s)ds-\int_t^T Z_sdB_s+K_T-K_t, \quad 0\le t\le T,\\ \\
Y_t\ge h(Y_t, \P_{Y_t}),\quad \forall t\in [0,T] \,\,\mbox{ and }\,\, \int_0^T (Y_t-h(Y_t,\P_{Y_t}))dK_t = 0.
\end{array}
\right.
$$
where $\P_{Y_t}$ is the marginal probability distribution of the process $Y$ at time $t$. Mean Field BSDEs have been introduced in \cite{blp09} and motivated by their connection to control of McKean Vlasov equation or Mean Field games. While the addition of a reflection to these BSDEs has already been considered in \cite{Li14}, our point of interest here is to allow the obstacle to depend on the distribution of the $Y$-component of the solution. Observe also that this class of Mean Field RBSDE shares some possibly fruitful proximity with the notion of averaged obliquely reflected BSDE discussed in \cite{ChaRic}, whenever $\P_{Y_.}$ reduces to $\E[Y_.]$.\\

In Section 2, we provide a clear formulation of the problem  and introduce the required assumptions on the coefficients, while focusing for ease of presentation on the simpler case where the distributional dependence with respect to $\P_{Y}$ boils down to $\E[Y]$ and the driver does not depend on $Z$. 
In Section 3 we state and prove the main results of the paper, namely Theorem \ref{FP-Snell} and \ref{thmp1}, dedicated to the solvability of mean-field reflected BSDEs. The existence of a unique solution is derived in terms of a fixed point argument for the associated Snell envelope of processes, and the extension to more general dependence with respect to $\P_{Y}$  is provided in Remark \ref{marginal}. 
In Section 4, we use an alternative approach and show convergence of the classical penalization scheme for BSDEs  to the solution of mean-field reflected BSDEs, under further monotone assumptions on the driver $f$ and the solvency constraint $h$. Such monotone property also ensures that the Skorohod type condition indeed induces the minimality property of the solution.  The more involved case where the driver also depends on $Z$ is finally discussed in Section 5.

\medskip

\subsection*{Notation}
Let $(\Omega,\mathcal{F},\P)$ be a complete probability space on which is defined a standard $d$-dimensional Brownian motion $B=(B_{t})_{0\le t\leq T}$. We denote by $(\mathcal{F}_{t}^{0}:= \sigma\{B_{s}, s\leq t\})_{0\le t\leq T}$ the natural filtration of $B$ and $\mathbb{F}:=(\mathcal{F}_{t})_{0\le t\leq T}$ its completion with the $\P$-null sets of $\Fc$. Let $\Pc$ be the $\sigma$-algebra on $\Omega\times [0,T]$ of $\mathcal{F}_{t}$-progressively measurable sets. Next, we introduce the following spaces. For $p\ge 1$, we let
\medskip
\begin{itemize}
\item[(i)]  $L^{p}=\{\eta:  \mathcal{F}_{T}$-measurable random variable such that $\E[|\eta|^{p}]<\infty\}$;\\
\item[(ii)] $\mathcal{H}^{p,m}=\{(v_{t})_{0\leq t \le T}:$ $\Pc$-measurable, $\R^m$-valued process such that $\E[(\int_{0}^{T}|v_{s}|^{2}\,ds)^{\frac{p}{2}}]<\infty\}$ ($m\ge 1$);\\
\item[(iii)] $\mathcal{S}^{p}=\{{(y_t)}_{0 \leq t \leq T}:$  $\Pc$-measurable process  such that $\E[\sup\limits_{0\leq t\leq T}|y_{t}|^{p}]<\infty\}$;\\
\item[(iv)] $\mathcal{S}_{c}^{p}=\{(y_{t})_{0\leq t\leq T}:$ continuous process of
$\mathcal{S}^{p}$\} which is a complete separable space. For $0\le s\le t\le T$, we set $\|y\|_{\sp([s,t])}:=(\E[\sup\limits_{s\leq u\leq t}|y_{u}|^{p}])^ {\frac{1}{p}}$.\\
\item[(v)] $\mathcal{S}_{i}^{p}=\{(k_{t})_{0\leq t\leq T}:$ continuous non-decreasing process
of  $\mathcal{S}^{p}$ such that $k_{0}=0$\};\\
\item[(vi)]$\mathcal{T}_{t}=\{\nu, \,\,\mathcal{F}_{t}$-stopping time such that $\mathbb{P}$-a.s. $\nu\geq t\}$; \\
\item[(vii)] $\Dc:=\{ (y_{t})_{0\leq t\leq T}:\,\, \mathbb{F}$-adapted $\R$-valued continuous process such that \newline  $\|y\|_1=\sup_{\tau \in {\Tc}_0}\E[|y_\tau|]<\infty\}.$ \\
The normed space $(\Dc, \|\cdot\|_1)$ is complete (see e.g. \cite{dm}, pp.90). For $0\le s\le t \le T$, we let $(\Dc([s,t]),\|\cdot\|_1)$ denote the complete metric space endowed with the norm $\|\cdot\|_1$ restricted to $[s,t]$:
 $$
\|X\|_1:=\sup_{\tau \in {\Tc}_0,\,\, s\le\tau\le t}\E[|X_\tau|]<\infty.
$$

\item[(viii)] For $q\in (0,1)$, we define \newline
$\Mc^q=\{(v_{t})_{0\leq t \le T}: \Pc$-measurable, $\R^m$-valued process such that \newline  $\|v\|_{\Mc^q}:=\E\left[\left(\int_0^T|v_s|^2ds\right)^{\frac{q}{2}}\right]<\infty\}.$ \\ 
The space $(\Mc^q,\|\cdot\|_{\Mc^q})$ is a complete metric space. The restriction of $\Mc^q$ to $[s,t]$ is denoted by $\Mc^q([s,t])$.

\end{itemize}
\medskip
\section{Formulation of the problem}
\subsection{The class of reflected mean-field BSDEs}
In this paper we propose to find solutions $(Y,Z,K)$ to the following class of reflected BSDE of mean-field type associated with the driver $f$, the terminal condition $\xi$ and the lower barrier $h$, in the cases $p>1$ and $p=1$ respectively, that we make precise in the following
\begin{Definition} \label{rbsde} We  say that the triple of progressively measurable processes $(Y_{t},Z_{t},K_{t})_{t \leq T}$ is a solution of the mean-field reflected BSDE associated with $(f,\xi,h)$ if, 
\begin{itemize}
\item [(1)] when $p>1$,  
\begin{equation}
 \label{Eq1}\left\{
    \begin{array}{ll}
      Y\in \Sc^{p}_c,\, Z\in \Hc^{p,d} \,\,\mbox{ and }\,\, K\in \Sc_{i}^{p},\\
       \displaystyle Y_{t}=\xi+\int_{t}^{T}f(s,Y_{s},\E[Y_s])\,
ds+K_{T}-K_{t}-\int_{t}^{T}Z_{s}\,dB_{s},\quad 0\le t\leq T,\\
      Y_{t}\geq h(Y_t,\E[Y_t]),\,\,\,  \forall t \in[0,T] \,\,\mbox{ and }\,\, \int_0^T (Y_t-h(Y_t,\E[Y_t]))dK_t = 0.
    \end{array}
  \right.
\end{equation}
\item[(2)] when $p=1$,
\begin{equation}
 \label{Eq1-p1}\left\{
    \begin{array}{lll}
      Y\in \Dc, Z\in \underset{q\in (0,1)}{\cup}\Mc^{q}\,\,\mbox{ and }\,\, K\in \Sc_{i}^{1},\\
       Y_{t}=\xi+\int_{t}^{T}f(s,Y_{s},\E[Y_s])\,
ds+K_{T}-K_{t}-\int_{t}^{T}Z_{s}\,dB_{s},\quad 0\le t\le T, \\
      Y_{t}\geq h(Y_t,\E[Y_t]),\,\,\,  \forall t \in [0,T] \,\,\mbox{ and }\,\,\int_0^T (Y_t-h(Y_t,\E[Y_t]))dK_t = 0.
    \end{array}
  \right.
\end{equation}
\end{itemize}
\end{Definition}

\medskip
\noindent  In order to alleviate the presentation, we focus in this Section on the particular case where the driver does not depend on $Z$. We present the general study in Section \ref{general} below as this case needs more involved techniques and more restrictive conditions on the coefficients. 
\medskip
Similarly, the main results of the paper, Theorems \ref{FP-Snell} and \ref{thmp1} below, which establish the solvability of the systems \eqref{Eq1} and \eqref{Eq1-p1} respectively, remain valid if we replace the mean-field coupling $\E[Y_t]$ with the more general marginal law coupling $\P_{Y_t}$ of $Y_t$, i.e. when the driver and the obstacle are of the form $f(Y_t,\P_{Y_t})$ and $h(Y_t,\P_{Y_t})$ (see Remark \ref{marginal} for further details). In particular, by taking the barrier $h$ of the form $h(Y_t, \P_{Y_t}):=Y_t-\E[\ell(t,Y_t)]$, observe that we obtain the class of  BSDE with mean reflection with loss function $\ell$ considered in \cite{briandetal} for which the solution $Y$ satisfies $\E[\ell(t,Y_t)]\ge 0$.

\medskip
Following \cite{elkaroui1}, the solution $Y$ of the mean-field reflected BSDE in Definition \ref{rbsde}, if it exists, is the Snell envelope of the process $L:=(L_t)_{t\le T}$ where
$$
L_t:=\int_0^t f(s,Y_s,
\E[Y_s])ds+h(Y_t,\E[Y_s]_{s=t})1_{\{t<T\}}+\xi1_{\{t=T\}},
$$
given by
\begin{equation}\label{MF-Snell}
Y_t=\underset{\tau\in\mathcal{T}_t}{\esssup\,}\E[\int_t^\tau f(s,Y_s,
\E[Y_s])ds+h(Y_\tau,\E[Y_s]_{s=\tau})1_{\{\tau<T\}}+\xi1_{\{\tau=T\}}|F_t],\,\,t\le
T.
\end{equation}

Driver and obstacles of mean-field type are common in life insurance contracts. 
A typical example is the case where the benefits less premiums include a cost of capital fee which is proportional to the reserve. Considering payments which involve a mean-field coupling such as  $\E[Y]$  comes from a very common practice among actuaries to consider the so-called 'model points' method  which is some sort of averaging of large homogeneous portfolios when computing reserves and designing policies. We refer to  \cite{cdd14} and \cite{bl14} for further details on reserve-dependent policies in life insurance and pensions. Here is an example of such a class of contracts.
 
 \medskip
$\bullet\,$ {\bf Guaranteed life endowment with a surrender/withdrawal option}.
Consider a  portfolio of a large number $N$ of homogeneous life insurance policies $\ell$. Denote by $(Y^{\ell,N}, Z^{\ell,N})$ the characteristics of the prospective reserve of each policy $\ell=1,\ldots,N$. We consider nonlinear reserving where the driver $f$ depends on the reserve for the particular contract and on the average reserve characteristics over the $N$ contracts (since $N$ is very large, averaging over the remaining $N-1$ policies has roughly the same effect as averaging over all $N$ policies): For each $\ell=1,\ldots,N$,
\begin{equation}\label{MF-N}\begin{array}{lll}
f(t,Y^{\ell,N}_t,(Y^{m,N}_t)_{m\neq\ell}):=\alpha_t-\delta_tY_t+\beta_t\max(\theta_t, Y^{\ell,N}_t-\frac{1}{N}\sum_{k=1}^NY^{k,N}_t),\\
h(Y^{\ell,N}_t,(Y^{m,N}_t)_{m\neq\ell})=u-c(Y^{\ell,N}_t)+\mu(\frac{1}{N}\sum_{k=1}^NY^{k,N}_t-u)^+,
\end{array}
\end{equation}
where $0<\mu<1$. The driver includes the discount rate $\delta_t$ and  deterministic positive functions   
$\alpha_t, \beta_t$ and $\theta_t$ which constitute the elements of withdrawal option. The solvency level $h$ is constituted of a required minimum or a benchmark return $u$, a reserve dependent management fee  $c(Y^{\ell,N}_t)$ (usually much smaller than $u$) and  a bonus option $(\frac{1}{N}\sum_{k=1}^NY^{k,N}_t-u)^+$ which is the possible surplus realized by the average of all involved contracts, which reflects the cooperative aspect of the pool of insurance contracts.

\noindent Sending $N$ to infinity in \eqref{MF-N}, yields the following forms of the driver and the obstacle: 
\begin{equation*}\begin{array}{lll}
f(t,Y_t,\E[Y_t]):=\alpha_t-\delta_tY_t+\beta_t\max(\theta_t, Y_t- \E[Y_t]),\\
h(Y_t, \E[Y_t])=u-c(Y_t)+\mu(\E[Y_t]-u)^+,
\end{array}
\end{equation*}
of the prospective reserve of a representative life insurance contract, a.k.a. the model-point among actuaries.

Pricing this type of contracts is one of the main motivations of studying the class \eqref{Eq1} of MF-reflected BSDEs, while clarifying the connection between such BSDE and Mean Field Games of timing problems \cite{Ber} is left for further research.

\subsection{Assumptions on $(f,h,\xi)$} 
We make the following assumption on $(f,h,\xi)$.
  
\medskip
\underline{\bf Assumption (H1)}. The coefficients $f, h$ and $\xi$ satisfy 
\begin{itemize}
\item[(i)] $f$ is a mapping from $[0,T]\times \Omega\times \R^2$ into $\R$ such that
\begin{itemize}
\item[(a)] the process $(f(t,0,0))_{t\le T}$ is $\Pc$-measurable and belongs to $\mathcal{H}^{p,1}$;
\item[(b)] $f$ is Lipschitz w.r.t. $(y,y')$ uniformly in $(t,\omega)$, i.e., there exists a positive constant $C_f$ such that  $\P$-$\as$ for all $t\in [0,T]$, 
$$
|f(t,y_1,y'_1)-f(t, y_2,y'_2)|\le C_f(|y_1-y_2|+|y'_1-y'_2|).
$$
for any $y_1,y'_1,y_2,y'_2\in\R$.
\end{itemize}

\item[(ii)]  $h$ is a mapping from $\R^{2}$ into $\R$ which is Lipschitz w.r.t. $(y,y')$, i.e., there exist two positive constants $\gamma_1$ and $\gamma_2$ such that for any $x,y,x'$ and $y'$
 $$
 |h(x,x')-h(y,y')|\le \gamma_1|x-y|+\gamma_2|x'-y'|
 $$
 where $\gamma_1$ and $\gamma_2$ are two positive constants.

\item[(iii)]  $\xi$ is an $\Fc_T$-measurable, $\R$-valued r.v., $\E[\xi^p]<\infty$ and satisfies
  $\xi\ge h(\xi,\E[\xi])$.

\end{itemize}
\bigskip

\begin{remark}
Observe that Assumption (H1) only contains  classical Lipschitz and integrability conditions on the coefficients together with a natural condition ensuring that the terminal condition satisfies the constraint of interest. 
\end{remark}

Under Assumption (H1), we first derive existence and uniqueness of the solution of \eqref{Eq1} (for $p>1$) and \eqref{Eq1-p1} (for $ p=1$) based on a fixed point argument, using  the notion of Snell envelope of processes. Then, under further monotonicity assumptions on $(f,h)$,  we show that the classical penalization scheme converges to that solution together with a minimality property of the solution satisfying the Skorohod type condition. 

\section{Existence of a unique solution via the  Snell envelope method}

 This Section is dedicated to the construction of a unique solution to the MF-reflected BSDE \eqref{Eq1} under Assumption $(H1)$, studying successively the cases where $p>1$ and $p=1$. Such result will require an additional smallness condition on the Lipschitz regularity of the constraint, see Relation \eqref{gamma} on the coefficients $\ga$ and $\gb$ below together with the discussion in Remark \ref{gagbpetit}. In both cases, our argumentation follows from a fixed point property of the Snell Enveloppe representation of the solution on small time intervals, combined with a proper pasting of the solutions on small intervals into a global solution on $[0,T]$. These results are provided below in Theorem \ref{FP-Snell} for $p>1$ and Theorem \ref{thmp1} for $p=1$, while the more general case where the driver and the constraints depend at time $t$ on the marginal distribution of $Y_t$ is treated in Remark \ref{marginal}.\\

Let $\Phi$ be the mapping that associates to a process $Y$ another process $\Phi(Y)$ defined by
\begin{equation}\lb{defi}\Phi(Y)_t:=\underset{\tau\in\mathcal{T}_t}{\esssup\,}\E\left[\int_t^\tau f(s,Y_s,
\E[Y_s])ds+h(Y_\tau,\E[Y_s]_{s=\tau})1_{\{\tau<T\}}+\xi1_{\{\tau=T\}}|F_t\right],\,\,t\le
T.
\end{equation}

We will show that the map $\Phi$ is well defined and admits a fixed point, first considering the situation where $p>1$ and then turning to the case where $p=1$. 

\subsection{The case $p>1$} $ $ 

In this case, we will establish in this Section that the map $\Phi$ has a unique fixed point on the complete metric space $\sp$.

Let first observe in the next Lemma that $\Phi$ is a well-defined map from $\sp$ to itself. 
\begin{lemma}\label{well} Let $f$, $h$ and $\xi$ satisfy Assumption (H1) for some $p>1$. If $Y\in \sp$ then $\Phi(Y)\in \sp$.
\end{lemma}
\begin{proof} Set, for $t\le T$,
$$\begin{array}{c}
L_t:=\int_0^t f(s,Y_s,
\E[Y_s])ds+h(Y_t,\E[Y_s]_{s=t})1_{\{t<T\}}+\xi1_{\{t=T\}}.\end{array}
$$
By construction, observe that 
\begin{equation}\label{envsnell}\begin{array}{c}
\Phi(Y)_t+ \int_0^tf(s,Y_s, \E[Y_s])ds =\underset{\tau\in \mathcal{T}_t}{\esssup\,}\E[L_{\tau}|\Fc_t]\;, \qquad 0\le t \le T\,.\end{array}
\end{equation} 
For $s\le T$, we linearize the mappings $f$ and $h$ as follows:
$$
\begin{array}{lll}
f(s,Y_s,\E[Y_s])=f(s,0,0)+a_f(s)Y_s+b_f(s)\E[Y_s],\\ h(Y_s,\E[Y_s])=h(0,0)+a_h(s)Y_s+b_h(s)\E[Y_s]
\end{array}
$$
where $a_f(\cdot), b_f(\cdot), a_h(\cdot)$ and $b_h(\cdot)$ are the adapted processes defined, for any $s\le T$, by
\begin{equation}\label{linearize}\left\{\begin{array}{lll}
a_f(s):=\frac{f(s,Y_s,\E[Y_s])-f(s,0,\E[Y_s])}{Y_s}\ind_{\{Y_s\neq 0\}},\quad a_h(s):=\frac{h(Y_s,\E[Y_s])-h(0,\E[Y_s])}{Y_s}\ind_{\{Y_s\neq 0\}}.
\\ \\  b_f(s):=\frac{f(s,0,\E[Y_s])-f(s,0,0)}{\E[Y_s]}\ind_{\{\E[Y_s]\neq 0\}},\quad b_h(s):=\frac{h(0,\E[Y_s])-h(0,0)}{\E[Y_s]}\ind_{\{\E[Y_s]\neq 0\}}.
\end{array}
\right.
\end{equation}
In view of the Lipschitz continuity of $f$ and $h$ provided by Assumption (H1), we have 

 $$
\max(|a_f(.)|,|b_f(.)|)\le C_f,\,\,\, |a_h(.)|\le \gamma_1,\,\,\, |b_h(.)|\le \gamma_2.
$$
We have, for any $(\Fc_t)$-stopping time $\nu$,
$$
\begin{array}{lll}
L_{\nu}=\int_0^\nu (f(s,0,0)+a_f(s)Y_s+b_f(s)\E[Y_s])ds\\ \qquad\qquad +
(h(0,0)+a_h(\nu)Y_\nu+b_h(\nu)\E[Y_s]_{s=\nu})1_{\{\nu<T\}}+\xi1_{\{\nu=T\}}.
 \end{array}
$$
Moreover, since $Y\in\sp$, the process $(M_t)_{0\le t\le T}$ defined by 
$$
\begin{array}{lll}
M_t:=\E\left[\int_0^T|f(s,0,0)|+C_f(|Y_s|+\E[|Y_s|])ds+
|h(0,0)|+\gamma_1\sup_{s\le T}|Y_s| \right. \\ \left. \qquad\qquad\qquad +\gamma_2\sup_{s\le
T}\E[|Y_s|]+|\xi||F_t\right],
\end{array}
$$
is a continuous martingale which, by Doob's inequality, belongs to $\sp$. We deduce that 
$$\begin{array}{lll}
|\E[L_{\tau}|F_t]|\le M_t.
 \end{array}
$$
for any $t\le T$ and $\tau\in\mathcal{T}_t$ and, in view of  \eqref{envsnell}, obtain
$$
|\Phi(Y)_t+\int_0^tf(s,Y_s,\E[Y_s])ds|\le M_t.
$$
Therefore,  
$$ 
\E[\sup_{t\le T}|\Phi(Y)_t|^p]\le
C_p\left\{\E\left[\left(\int_0^T|f(s,Y_s,\E[Y_s])|ds\right)^p\right]+\E[\sup_{t\le
T}|M_t|^p]\right\},$$
where $C_p$ is a positive constant that only depends on $p$ and $T$. Using once more the fact that $Y\in \sp$ and the Lipschitz property of the driver $f$, it finally holds that $\Phi(Y) \in \sp$. 
\end{proof}

We are now in position to obtain a contraction property of the map $\Phi$ on $\sp$ and first derive such property on a small time horizon.

\begin{proposition}\label{delta} Let Assumption (H1) hold for some $p>1$. If $\ga$ and $\gb$ satisfy 
\begin{equation*}
(\ga+\gb)^{\frac{p-1}{p}} \{(\frac{p}{p-1})^p\ga+\gb\}^{\frac{1}{p}}<1,
\end{equation*}
then there
exists $\delta >0 $ depending only on $p$, $C_f,\ga,\gb$ and $T$ such that $\Phi$
is a contraction on the time interval $[T-\delta, T]$.
\end{proposition}

\begin{proof} Let $Y,Y'\in \sp$. Then, for any $t\leq T$, we have
\begin{equation*}\begin{array} {lll}
|\Phi(Y)_t-\Phi(Y')_t|=|\underset{\tau\in\mathcal{T}_t}{\esssup\,}\E[\int_t^\tau
f(s,Y_s,
\E[Y_s])ds+h(Y_\tau,\E[Y_s]_{s=\tau})1_{\{\tau<T\}}+\xi1_{\{\tau=T\}}|\Fc_t]
 \\ \qquad \qquad\qquad\qquad\qquad -\,\,\underset{\tau\in\mathcal{T}_t}{\esssup\,}\E[\int_t^\tau f(s,Y'_s,
\E[Y'_s])ds+h(Y'_\tau,\E[Y'_s]_{s=\tau})1_{\{\tau<T\}}+\xi1_{\{\tau=T\}}|\Fc_t]|
\\  \qquad \qquad\qquad\qquad \le \underset{\tau\in\mathcal{T}_t}{\esssup\,}\E[\int_t^\tau |f(s,Y_s,
\E[Y_s])-f(s,Y'_s, \E[Y'_s])|ds \\ \qquad\qquad\qquad\qquad \qquad\qquad\qquad\qquad \qquad\qquad +|h(Y_\tau,\E[Y_s]_{s=\tau})-h(Y'_\tau,\E[Y'_s]_{s=\tau})|
|\Fc_t].
 \end{array}
\end{equation*}
Therefore, for any $t\le T$,
\begin{equation}\label{ineqevsn}\begin{array} {ll}
|\Phi(Y)_t-\Phi(Y')_t|\le \E[\int_t^T|f(s,Y_s, \E[Y_s])-f(s,Y'_s,
\E[Y'_s])|ds \\ \qquad \qquad\qquad \qquad\qquad\qquad +\ga \underset{t\le s\le T}{\sup\,}|Y_s-Y'_s|+ \gb \underset{t\le s\le T}{\sup\,}\E[|Y_s-Y'_s|] |\Fc_t].
 \end{array}
\end{equation}
Next, let $\d >0$ and $t\in [T-\d,T]$. From \eqref{ineqevsn} and
the Lipschitz property of $f$ we obtain
\begin{equation*}\label{ineqevsn2}\begin{array} {ll}
|\Phi(Y)_t-\Phi(Y')_t|\le (\d C_f+\ga)\E[\underset{T-\d \le s\le
T}{\sup}|Y_s-Y'_s|)|\Fc_t]+ (\d C_f+\gb)\underset{T-\d \le s\le
T}{\sup}\E[|Y_s-Y'_s|]
 \end{array}
\end{equation*}
which implies that 
\begin{align*}\label{ineqevsn2}\begin{array}{lll}
|\Phi(Y)_t-\Phi(Y')_t|^p\le (2\d C_f+\ga+\gb)^{p-1}\{(\d C_f+\ga)(\E[\underset{T-\d \le s\le
T}{\sup}|Y_s-Y'_s|)|\Fc_t])^p \\ \qquad\qquad\qquad\qquad\quad +(\d C_f+\gb)(\underset{T-\d \le s\le
T}{\sup}\E[|Y_s-Y'_s|])^p\},
\end{array}
\end{align*}
since the convexity relation $(a x_1+b x_2)^p\le (a +b)^{p-1}(a x_1^p+b x_2^p)$ holds for any non negative real constants $a$, $b$, $x_1$ and $x_2$.
Taking the expectation of the supremum over $t\in [T-\d,T]$ on both sides, and using Doob's inequality (only with one term, the other one is deterministic) together with Jensen's inequality, we obtain 
\begin{align*}
&\E[\underset{T-\d\le s\le T}{\sup\,}|\Phi(Y)_s-\Phi(Y')_s|^p] \\&\qquad \le(2\d C_f+\ga+\gb)^{p-1} \{(\frac{p}{p-1})^p(\d C_f+\ga)\E[\underset{T-\d \le s\le
T}{\sup}|Y_s-Y'_s|^p]+ \\&\qquad \qquad \qquad \qquad +(\d C_f+\gb))\E[\underset{T-\d \le s\le
T}{\sup}|Y_s-Y'_s|^p]\}.
 \end{align*}
Hence,
\begin{equation}\label{ineqevsn2}\begin{array} {lll}
\|\Phi(Y)-\Phi(Y')\|_{\sp (\d, T)}\le \Lambda(C_f,p,\ga,\gb)\times \|Y-Y'\|_{\sp(\d,T)},
 \end{array}
\end{equation}
where $\sp(\d,T):=\sp([T-\d, T])$ is the space $\sp$ restricted to 
$[T-\d, T]$ and 
$$
\Lambda(C_f,p,\ga,\gb)(\delta):=(2\d C_f+\ga+\gb)^{\frac{p-1}{p}} \{(\frac{p}{p-1})^p(\d C_f+\ga)+(\d C_f+\gb)\}^{\frac{1}{p}}.
$$
\noindent Now, it is easy to see that if $$(\ga+\gb)^{\frac{p-1}{p}} \{(\frac{p}{p-1})^p\ga+\gb\}^{\frac{1}{p}}<1\,,$$
we can find $\delta>0$ which only depends on $C_f$, $p$, $\ga$ and
$\gb$ and more specifically not on the terminal value $\xi$ such that
\begin{equation}\label{defdelta}
\Lambda(C_f,p,\ga,\gb)(\delta)<1\;.
\end{equation}
This directly implies that $\Phi$ is a contraction on $\sp {([T-\d, T])}$ and hereby has a fixed point $Y$ which satisfies, for any $t\in
[T-\d,T]$,
\begin{equation}\label{eqsnell}Y_t=\underset{\tau\in\mathcal{T}_t}{\esssup\,}\E[\int_t^\tau f(s,Y_s,
\E[Y_s])ds+h(Y_\tau,\E[Y_s]_{s=\tau})1_{\{\tau<T\}}+\xi1_{\{\tau=T\}}|\Fc_t].
\end{equation}
\end{proof}

 By repeatedly applying the fixed point argument of Proposition \eqref{delta} over adjacent time intervals of fixed length  $\delta$ and pasting the solutions,  we finally obtain existence of a unique solution in $\sp\times \mathcal{H}^{p,d}\times \mathcal{S}^p_i$  to the mean-field reflected BSDE \eqref{Eq1} over the whole time interval $[0,T]$, as shown in the following theorem. 
  
\begin{Theorem}\label{FP-Snell} Suppose that Assumption (H1) holds for some $p>1$. If
$\ga$ and $\gb$ satisfy
\begin{equation}\label{gamma}
(\ga+\gb)^{\frac{p-1}{p}} \{(\frac{p}{p-1})^p\ga+\gb\}^{\frac{1}{p}}<1,
\end{equation}
then the mean-field reflected BSDE \eqref{Eq1} has a unique solution $(Y,Z,K)\in \sp\times \mathcal{H}^{p,d}\times \mathcal{S}^p_i $.
\end{Theorem}

\begin{proof} According to the previous argumentation, recall that the constant $\d$ satisfying the inequality \eqref{defdelta} only depends on $p$, $C_f$, $\ga$ and $\gb$. Hence, the fixed point argument in Proposition \ref{delta} can be applied  for any terminal condition $\xi$ on a time interval of size $\delta$.\\ 
Taking $Y^0$  the unique fixed point of $\Phi$ in $\in\sp([T-\d, T])$, observe that $Y$ satisfies \eqref{eqsnell} on $[T-\d,T]$. In view of the link between the Snell envelope of processes and solutions of reflected BSDEs (see \cite{elkaroui1} for more details), there exist processes $(Z^0,K^0)\in \Hc^{2,d}\times
\Sc_{i}^{2}([T-\d, T])$ such that $K^0_{T-\d}=0$ and $(Y^0,Z^0,K^0)$ solves \eqref{Eq1} on
$[T-\d, T]$.\\
Applying the same reasoning on each time interval $[T-(i+1)\delta,T-i\delta]$ with a similar dynamics but terminal condition $Y^{i-1}_{T-i\delta}$ at time $T-i\delta$, we build recursively for $i=1$ to any $n$  a solution $(Y^i,Z^i,K^i)\in \sp\times\Hc^{2,d}\times\Sc_{i}^{2}$ of on each time interval $[T-(i+1)\delta,T-i\delta]$. Pasting properly these processes, we naturally derive a solution $(Y,Z,K)$ satisfying \eqref{Eq1} on the full time interval $[0,T]$.\\ 
As for the uniqueness of solution to \eqref{Eq1} on the full time interval $[0,T]$, observe that $Y$ is recursively uniquely defined on each time interval by the fixed point contraction, thanks to the Snell Envelope representation \eqref{eqsnell}. Hence, $Y$ and hereby $Z$ are uniquely defined on the full time interval $[0,T]$ while $K$ simply identifies to
$$
K_t=Y_0-Y_t-\int_{t}^{T}f(s,Y_{s},\E[Y_s])\,
ds+\int_{t}^{T}Z_{s}\,dB_{s}\;, 0\le t \le T\;. 
$$
Hence, \eqref{Eq1} admits a unique solution in $\sp\times\Hc^{2,d}\times\Sc_{i}^{2}$ on $[0,T]$.
\end{proof}

\begin{Remark}\lb{gagbpetit} We make the following observations on the sufficient condition \reff{gamma}. \\ 

$(a)$ Since $(\frac{p}{p-1})^p>1$, the inequality \eqref{gamma}  implies that $\ga+\gb<1$. The term $(\frac{p}{p-1})^p$ which inflates the Lipschitz constant $\ga$, due to the use of Doob's inequality, makes the condition \eqref{gamma} a bit heavy.  We will see in Theorem \ref{thmp1} (for the case $p=1$) that the solvability of the MF-BSDE \eqref{Eq1-p1} only requires  $\ga+\gb<1$, as Doob's inequality is not required for the proof.
\newline
$(b)$ Noting that \eqref{gamma}  also reads 
\begin{equation}\label{gamma2}(\frac{p}{p-1})^p\ga+\gb<\frac{1}{(\ga+\gb)^{p-1}},
\end{equation}
using the fact that  $\lim_{p\rightarrow \infty}(\frac{p}{p-1})^p=e$, if $\ga+\gb<1$ there exists $p$ for which \eqref{gamma} is satisfied, since, when $p\rightarrow \infty$,  the left-hand side converges to $e\ga+\gb$ while the 
the right-hand side diverges to $+\infty$. \newline
On the other hand, since $\lim_{p\rightarrow 1}(\frac{p}{p-1})^p=+\infty$, then \eqref{gamma} is satisfied only if $\ga$ very small which means that $h$ varies very slowly w.r.t the component $y$. \newline
 $(c)$ When the barrier $h$ does not depend on the mean-field term  $\E[Y]$ (i.e. $\gamma_2=0$), the fixed point argument through the Snell envelope may not be appropriate to obtain the solvability of the BSDE  unless $\gamma_1<\frac{p-1}{p}$. For example,  if $\frac{1}{2}<\gamma_1<1$, the BSDE with a standard driver $f$ i.e. independent of $\E[Y]$ and barrier $h(y)=\gamma_1y$ does not satisfy Theorem \ref{FP-Snell} for $p=2$. But if $\xi\ge 0$ and $f(t,\omega,y)\ge 0$, then any solution $Y$ of equation \eqref{eqsnell} is non-negative. Then the condition $Y_t\ge \gamma_1Y_t$ is equivalent to $Y_t\ge 0$, and by Corollary 3.7 in \cite{elkaroui1}, the corresponding BSDE has a unique solution.   \newline
 $(d)$ Whenever the solvency constraint only depends on $\E[Y_t]$, i.e. when $\ga=0$, observe that \reff{gamma} simply reduces to the condition $\gb<1$. This is quite natural as for example a linear constraint $Y_t \ge \gb \E[Y_t]$ would be automatically violated as soon as $\gb>1$.  \newline
 $(e)$In the particular case of linear constraint of the form $Y_t\ge \gamma_1 Y_t + h(\E[Y_t])$ with $\gamma_1\ge 0$, the condition \eqref{gamma} simply rewrites $\gamma_1+\gamma_2<1$ as the condition is simply equivalent to $Y_t\ge \frac{h(\E[Y_t])}{1-\gamma_1}$  
\end{Remark}


\subsection{The case $p=1$} $ $   

In this case we establish that the map $\Phi$ has a unique fixed point on the complete metric space $\Dc$. 

Following the same line of proof as in Lemma \ref{well}, we can verify that $\Phi$ is a well-defined map from $\Dc$ to itself.

\begin{lemma}\label{well-p1} Let $f$, $h$ and $\xi$ satisfy Assumption (H1) for $p=1$. If $Y\in \Dc$ then $\Phi(Y)\in \Dc$.
\end{lemma}

In the next proposition we show that $\Phi$ admits a local fixed point. 

\begin{proposition} \label{pest1} Let $f$, $h$ and $\xi$ satisfy Assumption (H1) for $p=1$ and suppose that $$\ga+\gb<1\,.$$ Then there exists $\d>0$, depending only on $C_f, \ga$ and $\gb$  and a
process $Y$ which belongs to $\Dc([T-\d,T])$, such that for any $t\in  [T-\d,T]$,
$$
\Phi(Y)_t=Y_t.
$$
\end{proposition}
\begin{proof} We are going to show  existence of such a $\d$ such that $\Phi$ is a contraction on
$(\Dc([T-\d,T]),\|\cdot\|_1)$. Indeed, let $Y$ and $Y'$ be two processes of $\Dc$ and let $\d$ be a positive constant to be determined later on. For any stopping time
$\theta$ such that $\P$-a.s. $T-\d\le \theta \le T$, we have
\begin{equation*}\begin{array} {lll}
|\Phi(Y)_{\theta}-\Phi(Y')_{\theta}|=|\underset{\tau\in\mathcal{T}_{\theta}}{\esssup\,}\E[\int_{\theta}^\tau
f(s,Y_s,
\E[Y_s])ds+h(Y_\tau,\E[Y_s]_{s=\tau})1_{\{\tau<T\}}+\xi1_{\{\tau=T\}}|\Fc_{\theta}]
 \\ \qquad\qquad\qquad\qquad -\,\,\underset{\tau\in\mathcal{T}_{\theta}}{\esssup\,}\E[\int_{\theta}^\tau f(s,Y'_s,
\E[Y'_s])ds+h(Y'_\tau,\E[Y'_s]_{s=\tau})1_{\{\tau<T\}}+\xi1_{\{\tau=T\}}|\Fc_{\theta}]|
\\   \le \underset{\tau\in\mathcal{T}_{\theta}}{\esssup\,}\E[\int_{\theta}^\tau |f(s,Y_s,
\E[Y_s])-f(s,Y'_s, \E[Y'_s])|ds+|h(Y_\tau,\E[Y_s]_{s=\tau})-h(Y'_\tau,\E[Y'_s]_{s=\tau})|
|\Fc_{\theta}].
 \end{array}
\end{equation*}
In view of the Lipschitz continuity of $f$ and $h$, we have
\begin{equation*}\begin{array} {lll}
\E[|\Phi(Y)_{\theta}-\Phi(Y')_{\theta}|]\\
\quad \le \sup_{\tau\in\mathcal{T}_{\theta}}\E[\int_{{\theta}}^\tau |f(s,Y_s,
\E[Y_s])-f(s,Y'_s, \E[Y'_s])|ds +|h(Y_\tau,\E[Y_s]_{s=\tau})-h(Y'_\tau,\E[Y'_s]_{s=\tau})|]\\
\quad \le \sup_{\tau\in\mathcal{T}_{\theta}}\E[\int_{T-\d}^T|f(s,Y_s,
\E[Y_s])-f(s,Y'_s, \E[Y'_s])|ds +|h(Y_\tau,\E[Y_s]_{s=\tau})-h(Y'_\tau,\E[Y'_s]_{s=\tau})|]
\\ 
\quad =(2\d C_f+\ga+\gb) \underset{T-\d\le \tau\le T}{\sup}\E[|Y_\t-Y'_\t|].
 \end{array}
\end{equation*}
Now, since $\theta$ is an arbitrary stopping time in $[T-\d,T]$, it holds that 
\begin{equation*}
\sup_{T-\d\le \tau \le T}\E[|\Phi(Y)_\t-\Phi(Y')_\t|]\le (2\d C_f +\ga+\gb)\sup_{T-\d\le \tau\le T}\E[|Y_\t-Y'_\t|].
 \end{equation*}
Now, it is immediate to see that, whenever $\ga+\gb<1$, there exists $\d>0$ only depending on $C_f$, $\ga$ and $\gb$ such that $2\d C_f +\ga+\gb<1$. This yields that the map
$\Phi$ is a contraction on the complete metric space $(\Dc[T-\d,T],\|\cdot\|_1)$. Hence, $\Phi$ has a unique fixed point, i.e., there exists $Y\in \Dc[T-\d,T]$, such that
for any $t\in [T-\d,T]$, $\Phi(Y)_t=Y_t$. 
\end{proof}

In the next proposition we show that the fixed point $Y$ of $\Phi$ on $(\Dc([T-\d,T]),\|\cdot\|_1)$  yields the existence of a unique local solution of \eqref{Eq1} in $(\Dc([T-\delta,T])$.

\begin{proposition} \lb{pest2}Let $f$, $h$ and $\xi$ satisfy Assumption (H1) for $p=1$ and suppose that 
\begin{equation}\label{gamma-p1}
\ga+\gb<1.
\end{equation} 
Then, there exist $\d>0$ only depending on $C_f, \ga$ and $\gb$, and a triplet of processes $(Y,Z,K)$ such that
\begin{equation}\label{Eq2-p1}
\left\{
    \begin{array}{lll}
      Y\in \Dc([T-\delta,T]), \,\,Z\in \underset{q\in (0,1)}{\cup}\Mc^{q}([T-\delta,T]),\\  K \mbox{ continuous nondecreasing such that}\, K_{T-\d}=0,\,\, \E[K_T]<\infty, \,\,\\
       Y_{t}=\xi+\int_{t}^{T}f(s,Y_{s},\E[Y_s])\,
ds+K_{T}-K_{t}-\int_{t}^{T}Z_{s}\,dB_{s},\quad T-\d\le t\le T, \\
      Y_{t}\geq h(Y_t,\E[Y_t]),\,\,\,  \forall t \in [T-\d,T] \,\,\mbox{ and }\,\,\int_{T-\d}^T (Y_t-h(Y_t,\E[Y_t]))dK_t = 0.
    \end{array}
  \right.
\end{equation}
Moreover, if $(\bar Y,\bar Z,\bar K)$ is another triple which satisfies \eqref{Eq2-p1},
then for $t\in [T-\d,T]$, $Y_t=\bar Y_t$, $K_t=\bar K_t$ and
$Z_t1_{\{T-\d\le t\le T\}}=\bar Z_t 1_{\{T-\d\le t\le T\}}$, $dt\otimes d\P$-a.e.
\end{proposition}

\begin{proof}  Let $\d$ be the positive constant and $Y$ the process defined in Proposition \ref{pest1} above. Then, for all $t\in [T-\d,T]$,
\begin{equation}\lb{expyenvsnell}
Y_t=\underset{\tau\in\mathcal{T}_t}{\esssup\,}\E[\int_t^\tau
f(s,Y_s,
\E[Y_s])ds+h(Y_\tau,\E[Y_s]_{s=\tau})1_{\{\tau<T\}}+\xi1_{\{\tau=T\}}|F_t].\end{equation}
Therefore, $(Y_t+\int_{T-\d}^tf(s,Y_s,\E[Y_s])ds)_{t\in [T-\d,T]}$
belongs to $\Dc$  since $Y\in\Dc$  and it is the Snell envelope of
the process $(\int_{T-\d}^tf(s,Y_s,
\E[Y_s])ds+h(Y_t,\E[Y_t])1_{\{t<T\}}+\xi1_{\{t=T\}})_{t\in
[T-\d,T]}$. Then, by applying the Doob-Meyer decomposition
(\cite{dm}, pp. 211), there exists a continuous martingale $(M_t)_{t\in
[T-\d,T]}$ and a non-decreasing process $(K_t)_{\in [T-\d,T]}$
$(K_{T-\d}=0)$ such that
$$\forall \,\, t\in [T-\d,T],\quad  Y_t+\int_{T-\d}^tf(s,Y_s,\E[Y_s])ds=M_t-K_t.$$
Moreover, $\E[K_T]+\E[|M_T|]<\infty$. Next, for $t\le T-\d$, we set $M_t:=Y_t=\E[M_{T-\d}|\Fc_t]$ and $K_t=0$. Then, by the Martingale Representation Theorem
there exists a $\Pc$-measurable process $(Z_t)_{t\in  [T-\d,T]}$ valued in $\R^d$, such that for any $t\in  [T-\d,T]$,
$$
M_t=M_{T-\d}+\int_{T-\d}^t Z_sdB_s.
$$
Thus, for any $t\in [T-\d,T]$,
$$
Y_{t}=\xi+\int_{t}^{T}f(s,Y_{s},\E[Y_s])\,
ds+K_{T}-K_{t}-\int_{t}^{T}Z_{s}\,dB_{s}\,\, \mbox{ and }\,\,Y_t\ge
h(Y_t,\E[Y_t]).$$ On the other hand,  since  $\E[|M_T|]<\infty$, we observe that $\E[\sup_{t\in  [T-\d,T]}|M_t|^q]<\infty$, for
any $q\in (0,1)$ (see
\cite{briandetal2}, pp.125). Therefore, the
Burkholder-Davis-Gundy inequality (\cite{ry}, pp. 161) implies that $Z\in \Mc^{q}([T-\d,T])$ for any $q\in (0,1)$ and we deduce that
$Z\in \cup_{q\in (0,1)}\Mc^{q}([T-\d,T])$. Finally, for
$t\in [T-\d,T]$, the random time defined by
$$
D_t:=\inf\{s\in  [T-\d,T], Y_s=h(Y_s,\E[Y_s])\}\wedge T,
$$
is a stopping time which realizes the essential supremum in
\eqref{expyenvsnell}, i.e., it is optimal. Thus, by the optimal
stopping properties we have that
$\int_t^{D_t}(Y_s-h(Y_s,\E[Y_s])dK_s=0$ for any $t\in [T-\d,T]$.
Consequently we have $\int_{T-\d}^T(Y_s-h(Y_s,\E[Y_s])dK_s=0$ (see\cite{elkaroui1}, pp. 717). \ms

We now prove uniqueness. Let $(\bar Y, \bar Z,\bar
K)$ be another triple that satisfies \eqref{Eq2-p1}. Therefore, $\bar
Y$ satisfies, for all $t\in [T-\d,T]$,
\begin{equation}\lb{expyenvsnellbar}
\bar Y_t=\underset{\tau\in\mathcal{T}_t}{\esssup\,}\E[\int_t^\tau
f(s,\bar Y_s, \E[\bar Y_s])ds+h(\bar Y_\tau,\E[\bar
Y_s]_{s=\tau})1_{\{\tau<T\}}+\xi1_{\{\tau=T\}}|F_t].\end{equation}
As $\bar Y$ belongs $\Dc([T-\d,T])$, then it is the fixed point of the mapping
$\Phi$ on $[T-\d,T]$, thus $\bar Y_t=Y_t$ for any $t\in [T-\d,T]$.
Now, the equations satisfied by $Y$ and $\bar Y$ imply that for any
$t\in [T-\d,T]$ $K_t=\bar K_t$ and $Z_t1_{\{T-\d\le t\le T\}}=\bar
Z_t 1_{\{T-\d\le t\le T\}}$, as claimed.
\end{proof}

Arguing now as  in Theorem \ref{FP-Snell}, we are able to paste solutions on small intervals and derive a global solution on any time interval $[0,T]$, which is the main result of this
section, together with Remark \ref{marginal} below:

\begin{Theorem} \label{thmp1}Let $f$, $h$ and $\xi$ satisfy
Assumption (H1) for $p=1$ and suppose that $\ga+\gb<1$.\\ Then, there is a unique
triple of processes $(Y,Z,K)$ such that
\begin{equation}
 \label{Eq3-p1}\left\{
    \begin{array}{l}
      Y\in \Dc,\,\,  Z\in \underset{q\in (0,1)}{\cup}\Mc^{q}, \,\, K\in \Sc^1_i,  \\
       Y_{t}=\xi+\int_{t}^{T}f(s,Y_{s},\E[Y_s])\,
ds+K_{T}-K_{t}-\int_{t}^{T}Z_{s}\,dB_{s},\quad \forall t\in [0,T], \\
      Y_{t}\geq h(Y_t,\E[Y_t]),\,\,\,  \forall t \in [0,T] \,\,\mbox{ and }\,\, \int_{0}^T (Y_t-h(Y_t,\E[Y_t]))dK_t = 0.
    \end{array}
  \right.
\end{equation}
\end{Theorem}
\begin{Remark} Recall that the process $Y$ which satisfies \eqref{Eq3-p1} also admits the representation \eqref{MF-Snell}.

\end{Remark}


\begin{remark}\label{marginal} Theorems \ref{FP-Snell} and \ref{thmp1} are still valid if we replace the mean-field coupling $\E[Y_t]$ with the  marginal law  $\P_{Y_t}$ of $Y$ at time $t$ in both the driver  $f$ and the obstacle $h$, i.e. consider a MF-BSDE driven by $f(Y_t,\P_{Y_t})$ reflected on $h(Y_t,\P_{Y_t})$ and with terminal condition $\xi$, provided that $f$ and $h$ are Lipschitz continuous w.r.t. the Wasserstein distance i.e.  Assumptions (i) (b) and (ii)  are replaced with the following. For any $y_1,y_2\in\R$ and $\mu,\nu\in \mathcal{P}_p(\R^d)$,
$$
\begin{array}{lll}|f(t,y_1,\mu)-f(t, y_2,\nu)|\le C_f(|y_1-y_2|+d_p(\mu,\nu)),\\ \\
|h(y_1,\mu)-h(y_2,\nu)|\le \gamma_1|x-y|+\gamma_2d_p(\mu,\nu),
 \end{array}
 $$
where $d_p(\cdot,\cdot)$ is the $p$-Wasserstein distance between two probability measures on the subset $\mathcal{P}_p(\R^d)$ of probability measures with finite $p$-th moment, formulated in terms of a coupling between two random variables $X$ and $Y$ defined on the same probability space:
\begin{equation*}\label{d-coupling}
d_p(\mu,\nu):=\inf\left\{\left(\E\left[|X-Y|^p\right]\right)^{1/p},\,\,\text{law}(X)=\mu,\,\text{law}(Y)=\nu \right\}.
\end{equation*}
In particular, we have, for any $0\le s\le u\le t\le T$,
\begin{equation}\label{coupling-ineq}
\sup_{u\in[s,t]}d_p\left(\P_{Y_u},\P_{Y^{\prime}_u}\right)\le \sup_{u\in[s,t]}(\E[|Y_u-Y^{\prime}_u|^p])^{1/p}
\end{equation}
from which we derive the following useful inequality
\begin{equation}\label{0-coupling-ineq}
\sup_{u\in[s,t]}d_p(\P_{Y_u},\delta_0)\le \sup_{u\in[s,t]}(\E[|Y_u|^p])^{1/p}. 
\end{equation}
Moreover, the linearization \eqref{linearize} of the mappings $f$ and $h$ used in the proof above is still valid. Indeed, for $s\le T$, we have
$$
\begin{array}{lll}
f(s,Y_s,\P_{Y_s})=f(s,0,\delta_0)+a_f(s)Y_s+b_f(s)d_p(\P_{Y_s},\delta_0)\\ h(Y_s,\P_{Y_s})=h(0,\delta_0)+a_h(s)Y_s+b_h(s)d_p(\P_{Y_s},\delta_0),
\end{array}
$$
where, in view of the Lipschitz continuity of $f$ and $h$ w.r.t. $(y,\mu)$, the coefficients 
$b_f(\cdot)$ and $b_h(\cdot)$ become, for any $s\le T$, 
$$
b_f(s):=\frac{f(s,0,\P_{Y_s})-f(s,0,\delta_0)}{d_p(\P_{Y_s},\delta_0)}\ind_{\{d_p(\P_{Y_s},\delta_0)\neq 0\}},\quad b_h(s):=\frac{h(0,\P_{Y_s})-h(0,\delta_0)}{d_p(\P_{Y_s},\delta_0)}\ind_{\{d_p(\P_{Y_s},\delta_0)\neq 0\}}.
$$
and satisfy together with $a_f(\cdot)$ and $a_h(\cdot)$
$$
\max(|a_f(.)|,|b_f(.)|)\le C_f,\,\,\, |a_h(.)|\le \gamma_1,\,\,\, |b_h(.)|\le \gamma_2.
$$
Thanks to this linearization and the inequality \eqref{0-coupling-ineq}, Propositions \ref{well} and \ref{well-p1} extend to the general case under consideration. Moreover, the inequality \eqref{coupling-ineq} makes the above fixed point argument used in Theorems \ref{FP-Snell} and \ref{thmp1} extend to this general case as well. 

\end{remark}


\section{Convergence of the penalization scheme}

We now present an alternative approach for the derivation of a unique solution  to \eqref{Eq1}, using a penalization type constructive argumentation, which may also reinterpret as a recursive scheme  reflecting at step $n$ the solution to the BSDE $Y^n$ on the obstacle generated by the solution $Y^{n-1}$. The monotonic convergence of the penalized BSDEs naturally relies on a comparison argumentation, which requires as for non-reflected MF BSDE \cite{blp09} additional monotone properties on the driver and the constraint, see Assumption (H2) right below. Under such assumption, we verify that the Skorohod type condition indeed characterizes the minimality property of the solution to the constrained BSDE. The result is besides derived in Theorem \ref{Theor} under an additional monotony property of the driver, which is in fact not necessary as explained in Corrolary \ref{resulsansmonotf}. For the sake of simplicity, we choose to restrict to the case $p=2$.

The following assumptions on the mappings $f$, $h$ and terminal value $\xi$ will be in force throughout this Section. 

\ms
\underline{\bf Assumption (H2)}\: $f$, $h$ and $\xi$ satisfy
\ms
\begin{itemize}
\item[(a)] Assumption (H1) for $p=2$.
\item[(b)] For any fixed $y$ (resp. $y'$), the mapping $y'\in \R\mapsto
h(y,y')$ (resp.
 $y\in \R\mapsto h(y,y')$ is non-decreasing.

\item[(c)] the mapping $y'\in \R \mapsto f(t,y,y')$ is non-decreasing, for
$t,y$ fixed.
\end{itemize}

\ms
\begin{Remark}
\noindent (i) The monotonicity conditions imposed on $f$ and $h$ in (H2)(b,c) are needed to be able to apply the comparison theorem for solutions of mean-field BSDEs (cf. \cite{blp09}).

\medskip

\noindent (ii) Since $h$ is Lipschitz, by (H2)-b),  linearizing $h$, we obtain,  for all $y$ and
$y'$,
\begin{equation}\label{ineqh}
(h(y,y'))^+\le h(0,0)^++\ga y^++\gb(y')^+.\end{equation}

\end{Remark}
Let us now introduce the following penalization scheme of the
equation \eqref{Eq1}. For $n\ge 0$, let $(Y^n,Z^n)$ be the pair of
processes defined as follows: $Y^0=\underbar Y$
 where $(\underbar Y,\underbar Z)$ is the solution of the following BSDE of mean field type (which exists according to Theorem 3.1. in \cite{blp09}):
  \begin{equation}
 \label{Eq2}\left\{
    \begin{array}{ll}
      \underbar Y\in \Sc^{2}_c,\quad \underbar Z\in \Hc^{2,d} ;\\
       \underbar Y_{t}=\xi+\int_{t}^{T}f(s,\underbar Y_{s},\E[\underbar Y_s])
ds-\int_{t}^{T}\underbar Z_{s}\,dB_{s}, \quad t\le T.
    \end{array}
  \right.
\end{equation}
Next, for $n\ge 1$, we define $(Y^n,Z^n)$ as the solution of the
following standard BSDE:
\begin{equation}
\label{Eq3}\left\{
    \begin{array}{lll}
      Y^n\in \Sc^{2}_c, Z^n\in \Hc^{2,d} ;\\
       \displaystyle Y^n_{t}=\xi+\int_{t}^{T}f(s,Y^n_{s},\E[Y^{n-1}_s])\,
ds+K^n_T-K^n_t-\int_{t}^{T}Z^n_{s}\,dB_{s},\,\,t\le T,
  \end{array}
  \right.
\end{equation}
where, for $n\ge 1$,  $K^n$ denotes the  process
$$ \begin{array}{c}
K^n_t:=n\int_0^t(Y^n_s-h(Y_s^{n-1},\E[Y_s^{n-1}]))^-ds,\quad t\le T.\end{array}
$$
\begin{Theorem} \label{Theor} Suppose that $(\ga,\gb)$ satisfies \eqref{gamma} with $p=2$ and  that Assumption (H2) is fulfilled. Suppose also that $y\in \R \mapsto f(t,y,y')$ is non-decreasing, for any
$t,y'$ fixed.\\  
Then, the sequence $(Y^n,Z^n,K^n)_{n\ge 0}$ converges in $\Sc^{2}_c\times \Hc^{2,d} \times \Sc^{2}$ to the unique solution $(
Y,Z,K)$ of \eqref{Eq1} i.e.
\begin{equation}
 \label{Eq1pen}\left\{
    \begin{array}{ll}
     Y\in \Sc^{2}_c, Z\in \Hc^{2,d} \mbox{ and } K\in \Sc_{ci}^{2};\\
       \displaystyle Y_{t}=\xi+\int_{t}^{T}f(s,Y_{s},\E[Y_s])\,
ds+K_{T}-K_{t}-\int_{t}^{T}Z_{s}\,dB_{s},\,\,t\leq T; \\
      Y_{t}\geq h(Y_t,\E[Y_t]),\,  \forall t \geq T \mbox{ and } \int_0^T (Y_t-h(Y_t,\E[Y_t]))dK_t = 0.
    \end{array}
  \right.
\end{equation}
Besides, the Skorohod condition ensures the minimality property of the solution. 
\end{Theorem}
\begin{Remark} The additional assumption on the monotone property of $f(t,.,y')$ for any $(t,y')$ is in fact not necessary for such result, as detailed in Corrolary \ref{resulsansmonotf} below
\end{Remark}

\begin{proof}: Observe first, that, since the solution $( \underbar Y, \underbar Z)$ of \eqref{Eq2} exists, then, by
induction, it follows that, for any $n\ge 1$, there is a unique solution $(Y^n,Z^n)$ of \eqref{Eq3}. We proceed to the proof of the convergence in the following five steps. We first verify that the sequence $(Y^n)_n$ is monotone and bounded, leading to a limit $\hat Y$ which is proved to satisfy the constraint in the second step. Next, we derive the continuity of $\hat Y$ and finally conclude on the dynamics of the limit of $(Y^n,Z^n,K^n)$ as well as on the minimality property.

\medskip

\underline{Step 1}. We have, for any $n\ge 0$, $Y^n\le Y^{n+1}\le
Y$. \ms

\nd We will show by induction that, for any $n\ge 0$, $Y^0\le
Y^1\le Y^n\le Y^{n+1}\le Y$. To begin with, recall that
$(Y,Z,K)$ satisfies
\begin{equation}
 \label{Eq1comp}\left\{
    \begin{array}{ll}
      \displaystyle Y_{t}=\xi+\int_{t}^{T}f(s,Y_{s},\E[Y_s])\,
ds+K_{T}-K_{t}-\int_{t}^{T}Z_{s}\,dB_{s},\,\,t\leq T; \\
      Y_{t}\geq h(Y_t,\E[Y_t]),\,  \forall t \geq T \mbox{ and } \int_0^T (Y_t-h(Y_t,\E[Y_t]))dK_t = 0.
    \end{array}
  \right.
\end{equation}
Now, consider the following standard reflected BSDE satisfied
by $(\bar Y, \bar Z ,\bar K)$:
\begin{equation}
 \label{Eq1compint}\left\{
    \begin{array}{ll}
      \displaystyle \bar Y_{t}=\xi+\int_{t}^{T}f(s,\bar Y_{s},\E[Y_s])\,
ds+\bar K_{T}-\bar K_{t}-\int_{t}^{T}\bar Z_{s}\,dB_{s},\,\,t\leq T; \\
      \bar Y_{t}\geq h(Y_t,\E[Y_t]),\,  \forall t \geq T \mbox{ and } \int_0^T (\bar Y_t-h(Y_t,\E[Y_t]))dK_t = 0.
    \end{array}
  \right.
\end{equation}
Since the solution of \eqref{Eq1compint} is unique, then  we obviously
have $\bar Y=Y$.

The fact that $Y^0\leq Y^1$ follows from the standard comparison theorem for solutions of BSDEs (see
e.g. \cite{elkaroui2}) since the penalization term is non negative. Next,
if for some $n\ge 1$ we have $Y^{n-1}\leq Y^n$, then using once more
the comparison theorem for standard BSDEs we get that $Y^{n}\leq Y^{n+1}$
since the mappings $y'\mapsto f(y,y')$, $y\mapsto h(y,y')$ and
$y'\mapsto h(y,y')$ are non-decreasing. Thus, for any $n\ge 0$, we have $Y^n\leq Y^{n+1}$.

Since the mapping $y'\mapsto f(y,y')$ is non-decreasing, by the comparison theorem for solutions of mean-field BSDEs in \cite{blp09} (see Theorem \ref{compappen} in Appendix), we have also $Y^0=\underbar Y\le Y=\bar Y$.

The mappings $y'\mapsto f(y,y')$, $y\mapsto h(y,y')$ and
$y'\mapsto h(y,y')$ being nondecreasing, we have
$$
\{f(s,y,\E[Y^{0}_s])+n(y-h(Y_s^{0},\E[Y_s^{0}])^-\}_{|y=\bar Y_s}=f(s,y,\E[Y^0_s])_{|y=\bar Y_s}\le
f(s,y,\E[Y_s])_{|y=\bar Y_s}.
$$
In view of the  comparison theorem for solutions of standard BSDEs we obtain that
$Y^1\le \bar Y$. But, $\bar Y$ is nothing else but $Y$. Therefore, $Y^1\le
Y$. Next, if for some $n\ge 1$, we have $Y^n\le \bar Y$, we use the
same argument to show that $Y^{n+1}\le \bar Y=Y$. Therefore, for any
$n\ge 0$, $Y^n\le Y$. \ms

\nd Summing up,  for any $n\ge 0$, $\,Y^n\le Y^{n+1}\le Y$. \bs

 \und{Step 2}. The limit $\hy$ satisfies the solvency constraint. \\
 Set, for $t\in [0,T]$, 
$$
\hat Y_t:=\lim_{n\rightarrow\infty}Y^n_t.
$$
Then the process $\hy$ is rcll and satisfies
\begin{equation}\lb{haty2}
\E[\sup_{t\le T}|\hy_t|^2]<\infty\quad  \mbox{ and }\quad \forall t\le
T,\quad \hy_t\ge h(\hy_t,E[\hy_t]).
\end{equation}
Indeed, the existence of the process $\hy$ is due to the fact that the
sequence $(Y^n)_{n\ge 0}$ is non decreasing. Moreover, since $Y^0\le \hy \le Y$, the square
integrability of $\hy$ is due to the square-integrability of $Y^0$ and $Y$. 

Next, for $n\ge 1$, let us set $\bar Y^n=Y^n-Y^0$. Therefore, $\bar
Y^n$ satisfies, for all $t\leq T$,
$$
\begin{array}{ll}
\bar Y^n_t=\int_{t}^{T}\{f(s,Y^n_{s},\E[Y^{n-1}_s])-f(s,\underbar
Y_{s},\E[\underbar Y_s])\}ds  \\ \qquad\qquad\quad 
\qquad \qquad \qquad +n\int_t^T(Y^n_s-h(Y_s^{n-1},\E[Y_s^{n-1}]))^-ds-\int_{t}^{T}(Z^n_{s}-\underbar
Z_{s})dB_{s}. \end{array}
$$
Since $f$ is nondecreasing w.r.t $y$ and $y'$, and we have $Y^n\ge \underbar Y$, then for any $s\le T$, $f(s,Y^n_{s},\E[Y^{n-1}_s])-f(s,\underbar Y_{s},\E[\underbar Y_s])\ge 0$.
This in turn  implies that for any $n\ge 0$, $\bar Y^n$ is a continuous
supermartingale which converges increasingly to $\hy -Y^0$,
therefore $\hy-Y^0$ is rcll (\cite{dm}, pp. 86). Consequently, $\hy$
is rcll, since $Y^0$ is continuous. Next, by \eqref{Eq3}, we have
$$
\E[\int_t^T(Y^n_s-h(Y_s^{n-1},\E[Y_s^{n-1}]))^-ds]=
\frac{1}{n}\E[Y^n_{0}-\xi-\int_{t}^{T}f(s,Y^n_{s},\E[Y^{n-1}_s])
ds].
$$
By Fatou's Lemma, we have
$$\E[\int_t^T(\hy _s-h(\hy_s,\E[\hy_s]))^-ds]\leq
\liminf_{n\rightarrow
\infty}\frac{1}{n}\E[Y^n_{0}-\xi-\int_{t}^{T}f(s,Y^n_{s},\E[Y^{n-1}_s])
ds]=0
$$
since $Y^0\le Y^n\le Y$ and the processes $Y^0, Y$ belong to $\mathcal{
S}^2_c$. As $\hy$ and $h(\hy,\E[\hy])$ are rcll, then for any $t<T$,
$\hy_t\geq h(\hy_t,\E[\hy_t])$. Finally, the inequality holds also on
$T$ since $\hat Y_T=\xi$ and $\xi\ge h(\xi,\E[\xi])$ by (H1)-iii).
\bigskip

\und{Step 3}. $\hy$ is continuous. 

\nd First note that by the uniform square integrability of $\hy$,
the process $(h(\hy_t,\E[\hy_t])_{t\le T}$ is uniformly square integrable since $h$ is Lipschitz
w.r.t $(y,y')$. On the other hand the process $Y^n$ has the
following representation as a Snell Envelope of processes: for all $t\in [0,T]$,
\begin{equation*}\lb{snellenvelopeyn}
Y^n_{t}=\underset{\tau \ge t}{\esssup}\E[\int_t^{\tau}f(s,Y^{n}_s,\E[Y^{n-1}_s])ds+Y^n_\tau\wedge
h(Y^{n-1}_\tau,\E[Y^{n-1}_t]_{|t=\tau})1_{\{\tau<T\}}+\xi
1_{\{\tau=T\}}|\Fc_{t}],
\end{equation*}
since
$$\begin{array}{c}
Y^n\geq Y^n\wedge h(Y^{n-1},\E[Y^{n-1}]) \mbox{ and }\int_0^T(Y^n_s-
Y_s^n\wedge h(Y_s^{n-1},\E[Y_s^{n-1}]))dK^n_s=0.\end{array}
$$ One can see the paper by El-Karoui et al. \cite{elkaroui1} for more details.
Now, if
for any $n\ge 0$, $U^n$ and $U$ are rcll and $U^n\nearrow U$ then
$\mbox{SN}(U^n)\nearrow \mbox{ SN}(U)$ (see e.g. \cite{djehmdpop} for more
details). It follows that, for any $t\le T$,
\begin{equation}\lb{snellenvelopey}
\hat Y_{t}=\underset{\tau \ge t}{\esssup}\E[\int_t^{\tau}f(s,\hat
Y_s,\E[\hat Y_s])ds+ h(\hat Y_\tau,\E[\hat
Y_t]_{|t=\tau})1_{\{\tau<T\}}+\xi 1_{\{\tau=T\}}|\Fc_{t}]
\end{equation}
since $Y^n\nearrow \hat Y$, $f$ and $h$ are increasing in their
arguments and finally $\hat Y\ge h(\hat Y,E[\hat Y])$. \ms

\nd Next, for $t\le T$, let us set  $\Phi(t):=\E[\hat Y_t]$ and 
show that $\Phi$ is continuous. 

\nd Indeed, first note that since the process $(\hat Y_t+\int_0^{t}f(s,\hat
Y_s,\E[\hat Y_s])ds)_{t\leq T}$ is a supermartingale, the
possible jumps of the process $\hat Y$ are only negative. Suppose
now that there exists $t$ such that $\Delta \Phi(t)<0$. Then  $\P\{\omega, \Delta \hat Y_t(\omega)<0\}>0$ and $\Delta
\Phi(t)=\E[\Delta\hat  Y_t]$. But \eqref{snellenvelopey} and the
characterization of the jumps of the Snell envelope (see e.g.
\cite{hmd2}) imply that $$-\Delta \hat Y_t=\lim_{s\nearrow t}(h(\hat
Y_s,\E[\hat Y_s])-h(\hat Y_t,\E[\hat Y_t]))\le \gamma_1(-\Delta \hat
Y_t)+\gamma_2(-\Delta \Phi(t)).$$ Take now the expectation in both
hand-sides to obtain that
$$-\Delta \Phi(t)\le
\gamma_1(-\Delta \Phi(t))+\gamma_2(-\Delta \Phi(t)).$$ But, this
contradicts the inequality $\gamma_1+\gamma_2<1$ (see \eqref{gamma} for $p=2$ and Remark \ref{gagbpetit}). 
Therefore, $\Phi$ is
continuous.

We now show that $\hat Y$ is continuous. Suppose there exists a
stopping time $\sigma$ such that $\Delta \hat Y_\sigma<0$. Then the
characterization of jumps of a Snell envelope of processes imply
that
$$
-\Delta \hat Y_\sigma\leq h(\hat Y_{\sigma-},\E[\hat
Y_t]_{t=\sigma})- h(\hat Y_\sigma,\E[\hat Y_t]_{t=\sigma})\leq
\gamma_1 (-\Delta \hat Y_\sigma)$$ since $\Phi(t)=\E[\hat Y_t]$ is
continuous. The last inequality is absurd since $\gamma_1<1$. Thus, for any stopping 
time $\sigma$, $\Delta \hat Y_\sigma=0$ i.e. $\hat Y$ is continuous.
\bigskip

\und{Step 4}. The sequence $(Y^n,Z^n,K^n)_{n\ge 0}$ converges to
$(\hy,\hz,\hk)$ solution of the reflected BSDE \eqref{Eq1pen}.

First note that since $\hy$ is continuous then, by Dini's theorem,
the convergence of $(Y^n)_{n\ge 0}$ to $\hy$ is uniform on $[0,T]$, i.e.,
$$
\P\mbox{-a.s.}, \quad \lim_{n\rightarrow \infty}\sup_{t\leq
T}|Y^n_t-\hy_t|=0.
$$
Next, the characterization \eqref{snellenvelopey} implies the
existence of a martingale $\hat M$ and a continuous increasing
process $\hat K$ ($\hat K_0=0$) such that
 $$
\hat Y_{t}+\int_0^{t}f(s,\hat Y_s,\E[\hat Y_s])ds=\hat M_t-\hat K_t,\quad 
t\leq T.$$ 
Moreover, in view of \eqref{haty2}, we have
$$
\E[(\hat K_T)^2]<\infty \quad \mbox{ and then }\quad \E[\sup_{t\leq T}|\hat M_t|^2]<\infty.
$$
The rest of the proof is classical i.e the martingale representation
provides the process $\hat Z$ and $\hat K$ satisfies the Skorohod
condition, i.e., $(\hat Y,\hat Z,\hat K)$ is a solution of the mean-field reflected BSDE \eqref{Eq1} (see \cite{elkaroui1} for
more details).
\bigskip

\und{Step 5}. The minimality property\\ 
 Lastly, let us show that this solution is minimal. Let $(Y',Z',K')$ be another solution of \eqref{Eq1} without imposing the Skorohod type condition. Then, by comparison of solutions
of mean-field BSDEs (see Theorem \ref{compappen} below) we obtain $Y^0\le Y'$. Next, assume that for some $n\ge 0$, we have $Y^n\le Y'$. Once more by comparison we obtain $Y^{n+1}\le Y'$ since 
\begin{align}
f(t,y,&\E[Y^n_t])+{(n+1)(y-h(Y^n_t, \E[Y^n_t]))^- }_{|_{y=Y'_t}}\nn\\&
=f(t, Y'_t, \E[Y^n_t])\le f(t, Y'_t,\E[Y'_t])\le f(t, Y'_t,\E[Y'_t])+dK'_t. \nn
\end{align}
Thus, for any $n\ge 0$, $Y^n\le Y'$ which implies, in taking the limit, that $Y\le Y'$ and then $Y$ is the minimal solution. 
\end{proof}
\begin{corollary}\lb{resulsansmonotf}Assume that $(\ga,\gb)$ satisfies \eqref{gamma} and $f$, $h$ and $\xi$ satisfy Assumption (H2).\\
Then, the sequence $(Y^n,Z^n,K^n)_{n\ge 0}$ converges in
 $\Sc^{2}_c\times \Hc^{2,d}\times \Sc_{c}^{2}$ to the unique
solution of the mean-field RBSDE \eqref{Eq1}  and the Skorohod condition induces the minimality of the solution.
\end{corollary}
\begin{proof} Let $\theta$ be a real constant and  $(Y,Z,K)$ be a solution of \eqref{Eq1}. For $t\le T$, set
$$
\unb Y_t=e^{\theta t}Y_t,\quad  \unb Z_t=e^{\theta t}Z_t \quad \text{ and }\quad \unb K_t=e^{\theta t}K_t.
$$
Then, $(\unb Y,\unb Z,\unb K)$ is  a solution of the following mean-field RBSDE:
\begin{equation}
 \label{Eq1equiv}\left\{
    \begin{array}{ll}
      \unb Y\in \Sc^{2}_c,\quad \unb Z\in \Hc^{2,d}, \quad \unb K\in \Sc_{i}^{2};\\
       \unb  Y_{t}=e^{\theta T}\xi+\int_{t}^{T}\unb F(s,\unb Y_{s},\E[\unb Y_s])\,
ds+\unb K_{T}-\unb K_{t}-\int_{t}^{T}\unb Z_{s}\,dB_{s},\quad  t\geq T; \\
      \unb Y_{t}\geq e^{\theta t}h(e^{-\theta t}\unb Y_t,e^{-\theta t}\E[\unb Y_t]),\,  \forall t \leq T, \\   \int_0^T (\unb Y_t-e^{\theta t}h(e^{-\theta t}\unb Y_t,e^{-\theta t}\E[\unb Y_t]))d\unb K_t = 0,
    \end{array}
  \right.
\end{equation}
where
$$
\unb F:(t,y,y')\mapsto e^{\theta t}f(t,e^{-\theta t}y,e^{-\theta t}y') -\theta y.
$$
But, by choosing $\theta$ appropriately we make that the function $\unb F$ satisfy all the assumptions (H2). Therefore, by
Theorem \ref{Theor}, its penalization scheme $(\unb Y^n,\unb
Z^n,\unb K^n)$, $n\ge 0$, defined similarly as in \eqref{Eq3}, 
converges in $\Sc^{2}_c\times \Hc^{2,d}\times \Sc_{c}^{2}$ to $(\unb Y,\unb Z,\unb K)$ the unique solution of \eqref{Eq1equiv}. But, by uniqueness, for any $t\le T$,
$$
\unb Y^n_t=e^{\theta t}Y^n_t,\quad  \unb Z^n_t=e^{\theta t}Z^n_t
\quad \text{ and }\quad d\unb K^n_t=e^{\theta t}dK^n_t.
$$
 Therefore $(Y^n,Z^n,K^n)_{n\ge 0}$ converges in
 $\Sc^{2}_c\times \Hc^{2,d}\times \Sc_{c}^{2}$ to $(Y,Z,K)$, the
 unique solution of \eqref{Eq1}. 
\end{proof}
\begin{remark}The constraint in (H2)-b) related to the monotonicity of $y\mapsto h(y,y')$ can also be relaxed substantially. Indeed, for $\kappa>0$, let us set
$$
\Psi_\kappa: (y,y')\mapsto \frac{1}{1+\kappa \ga}h(y,y')+\frac{\kappa \ga}{1+\kappa \ga}y.
$$
Observe, that, for any $\kappa>1$, the mapping $y\mapsto \Psi_\kappa (y,y')$ is non decreasing w.r.t $y$ even when $y\mapsto h(y,y')$  does not enjoy any specific property of monotonicity. On the other hand, the condition $y\ge h(y,y')$ is equivalent to $y\ge \Psi_\kappa (y,y')$, and the corresponding Skorohod conditions also coincide. Therefore as soon as the Lipschitz constants of $\Psi_\kappa$ verify the condition \eqref{gammatr}, i.e.,
\begin{equation}\label{gammatr}
(\frac{\ga+\kappa\ga}{1+\kappa\ga}+\frac{\gb}{1+\kappa\ga})^{\frac{p-1}{p}} \{(\frac{p}{p-1})^p\frac{\ga+\kappa\ga}{1+\kappa\ga}+\frac{\gb}{1+\kappa\ga}\}^{\frac{1}{p}}<1,
\end{equation}the penalization scheme of the MFBSDE associated with coefficients $(f,\xi,\Psi_\kappa)$ converges to the solution of the MFBSDE of interest associated with $(f,\xi,h)$.
\end{remark}
\section{A more general case: $f$ further depends on $z$} \label{general} 

\noindent The last Section of the paper is dedicated to the more involved framework, where the driver $f$ may also depend on the $Z$ process and, once again, in this Section we restrict to the case where $p=2$, for the sake of simplicity.

\noindent The following assumptions on the mappings $f$, $h$ and terminal value $\xi$ will be assumed hereafter.

\noindent \underline{\bf Assumption (H3)}. The coefficients $f, h$ and $\xi$ satisfy the following set of assumptions. 

\begin{itemize}

\item[(i)] $f$ is a mapping from $[0,T]\times \Omega\times \R^{2+d}$ into $\R$ such that
\begin{itemize}
\item[(a)] the process $(f(t,0,0,0))_{t\le T}$ is $\Pc$-measurable and belongs to $\mathcal{H}^{2,1}$;

\item[(b)] $f$ is Lipschitz w.r.t. $(y,z,y')$ uniformly in $(t,\omega)$, i.e., there exists a positive constant $C_f$ such that  $\P$-$\as$ for all $t\in [0,T]$,
$$
|f(t,y_1,z_1,y'_1)-f(t, y_2,z_2,y'_2)|\le C_f(|y_1-y_2|+|z_1-z_2|+|y'_1-y'_2|).
$$
for any $y_1,y'_1,y_2,y'_2\in\R$ and $z_1,z_2$ in $\R^d$.
\item[(c)] the mapping $y'\in \R \mapsto f(t,y,z,y')$ is non-decreasing, for
$t,y,z$ fixed.
\item[(d)] The domination condition: There exists a measurable function $\Phi(t,\omega,y,y')$ from $[0,T]\times \Omega \times \R^{2}$ into $\R^+$ such that: i) $\Phi$ is
Lipschitz in $(y,y')$ uniformly w.r.t $(t,\omega)$ ; ii) the process $(\Phi(t,\omega,0,0))_{t\le T}$ belongs to ${\mathcal H}^{2,1}$ ; 
iii) $\P
$-a.s. for any $t,y,y',z$ we have
$$
f(t,\omega,y,z,y')\le \Phi(t,\omega,y,y').
$$
\end{itemize}
\item[(ii)]  
\begin{itemize}
\item[(a)] $h$ is a mapping from $\R^{2}$ into $\R$ which is Lipschitz w.r.t. $(y,y')$, i.e., there exist two positive constants $\gamma_1$ and $\gamma_2$ such that for any $x,y,x'$ and $y'$
 $$
 |h(x,x')-h(y,y')|\le \gamma_1|x-y|+\gamma_2|x'-y'|,
 $$
 where $\gamma_1$ and $\gamma_2$ are two positive constants.
 \item[(b)] For any fixed $y$ (resp. $y'$), the mapping $y'\in \R\mapsto
h(y,y')$ (resp.
 $y\in \R\mapsto h(y,y')$ is non-decreasing.
\end{itemize}
\item[(iii)]  $\xi$ is an $\Fc_T$-measurable, $\R$-valued r.v., $\E[\xi^2]<\infty$ and satisfies
  $\xi\ge h(\xi,\E[\xi])$.
\end{itemize}
We have the following result.
\begin{Theorem} \label{Theor2} Suppose that $(\ga,\gb)$ satisfies \eqref{gamma} with $p=2$ and  that Assumption (H3) is fulfilled. Then there exists a triplet of processes
$(Y,Z,K)$ which is the minimal solution of the following mean-field reflected BSDE:
\begin{equation}
 \label{Eqz}\left\{
    \begin{array}{ll}
Y\in \Sc^{2}_c, Z\in \Hc^{2,d} \mbox{ and } K\in \Sc_{ci}^{2};\\
      Y_{t}=\xi+\int_{t}^{T}f(s,Y_{s},Z_s,\E[Y_s])\,
ds+ K_{T}- K_{t}-\int_{t}^{T}Z_{s}\,dB_{s},\,\,t\leq T; \\
      Y_{t}\geq h( Y_t,\E[ Y_t]),\,  \forall t \geq T \,\, \mbox{ and } \,\, \int_0^T ( Y_t-h( Y_t,\E[ Y_t]))d K_t = 0.
    \end{array}
  \right.
\end{equation}
\end{Theorem}
\begin{proof} Let $Y^0=\underbar Y$
where $(\underbar Y,\underbar Z)$ is the solution of the following BSDE of mean field type (which exists according to Theorem 3.1. in \cite{blp09}):
  \begin{equation}
 \label{Eq21}\left\{
    \begin{array}{ll}
      \underbar Y\in \Sc^{2}_c,\quad \underbar Z\in \Hc^{2,d} ;\\
       \underbar Y_{t}=\xi+\int_{t}^{T}f(s,\underbar Y_{s},\underbar Z_{s},\E[\underbar Y_s])
ds-\int_{t}^{T}\underbar Z_{s}\,dB_{s}, \quad t\le T.
    \end{array}
  \right.
\end{equation}
For $n\ge 1$, we define $(Y^n,Z^n)$ as the solution of the
following standard BSDE:
\begin{equation}
\label{Eq22}\left\{
    \begin{array}{lll}
      Y^n\in \Sc^{2}_c, Z^n\in \Hc^{2,d} ;\\
      Y^n_{t}=\xi+\int_{t}^{T}f(s,Y^n_{s},Z^n_{s},\E[Y^{n-1}_s])\,
ds+K^n_T-K^n_t-\int_{t}^{T}Z^n_{s}\,dB_{s},\,\,t\le T,
  \end{array}
  \right.
\end{equation}
where, for $n\ge 1$,  $K^n$ denotes the  process
$$ \begin{array}{c}
K^n_t:=n\int_0^t(Y^n_s-h(Y_s^{n-1},\E[Y_s^{n-1}]))^-ds,\quad t\le T.\end{array}
$$
Finally, let $(\bar Y,\bar Z,\bar K)$ be the unique solution of the following mean-field reflected BSDE:
\begin{equation}
 \label{Eq23}\left\{
    \begin{array}{ll}
\bar Y\in \Sc^{2}_c, \bar Z\in \Hc^{2,d} \mbox{ and }\bar K\in \Sc_{ci}^{2};\\
      \bar Y_{t}=\xi+\int_{t}^{T}\Phi(s,\bar Y_{s},\E[\bar Y_s])\,
ds+ \bar K_{T}-\bar  K_{t}-\int_{t}^{T}\bar Z_{s}\,dB_{s},\,\,t\leq T; \\
     \bar  Y_{t}\geq h(\bar  Y_t,\E[\bar  Y_t]),\,  \forall t \geq T \mbox{ and } \int_0^T ( \bar Y_t-h( \bar Y_t,\E[\bar  Y_t]))d\bar  K_t = 0.
    \end{array}
  \right.
\end{equation}
The triplet of processes $(\bar Y,\bar Z,\bar K)$ solution of \eqref{Eq23} exists by Theorem \ref{FP-Snell}.
\ms

\nd The proof of the theorem will be obtained in the following steps.
\ms

\nd \underline{Step 1}. \begin{equation}\lb{estimgeneral}\mbox{For
any }n\ge 0, \quad Y^0\le Y^{n}\le Y^{n+1}\le \bar Y.\end{equation}

\nd The fact that $Y^0\le Y^{1}$ is obtained by the comparison result by using the standard argument of It\^o's formula with $((Y^0-Y^1)^+)^2$ since the driver of
$Y^1$ is $f(t,y,z,\E[Y^0_t])+(y-\E[Y^0_t]])^+$ which is greater than $f(t,y,z,\E[Y^0_t])$. Actually, we can easily show that $\E[((Y^0-Y^1)^+)^2]=0$, which is the desired result. Suppose that, for some $n\ge 0$, we have $Y^{n}\le Y^{n+1}$. Then we also have  $Y^{n+1}\le Y^{n+2}$ by the comparison theorem of solutions since the mappings $y'\in \R\mapsto f(t,y,z,y')$, $y'\in \R\mapsto h(y,y')$ and $y\in \R\mapsto h(y,y')$ are nondecreasing. Finally, let us show by induction that for any $n\ge 0$, $Y^n\le \bar Y$. For $n=0$, it holds true by the comparison theorem of solutions of mean-field BSDEs since the function $y'\in \R\mapsto f(t,y,z,y')$ is non-decreasing and $d\bar K\ge 0$ (see Theorem \ref{compappen} in the appendix). Next, assume that the property holds true for some $n\ge 0$, i.e. $Y^n\le \bar Y$. But the solution of \eqref{Eq23} is unique, then $(\bar Y,\bar Z,\bar K)=(\tilde Y,\tilde Z,\tilde K)$ where $(\tilde Y,\tilde Z,\tilde K)$ is the unique solution of the following standard BSDE:
\begin{equation}
 \label{Eq232}\left\{
    \begin{array}{ll}
\tilde Y\in \Sc^{2}_c, \tilde Z\in \Hc^{2,d} \mbox{ and } \tilde K\in \Sc_{ci}^{2};\\
      \tilde Y_{t}=\xi+\int_{t}^{T}\Phi(s,\tilde  Y_{s},\E[\bar Y_s])\,
ds+ \tilde  K_{T}-\tilde   K_{t}-\int_{t}^{T}\tilde  Z_{s}\,dB_{s},\,\,t\leq T; \\
    \tilde   Y_{t}\geq h(\bar   Y_t,\E[\bar  Y_t]),\,  \forall t \geq T \mbox{ and } \int_0^T ( \tilde Y_t-h(\bar  Y_t,\E[\bar  Y_t]))d\tilde   K_t = 0.
    \end{array}
  \right.
\end{equation}
Now, by the induction hypothesis and the fact that the mappings,
$y'\in \R\mapsto \Phi(t,y,y')$, $y'\in \R\mapsto f(t,y,z,y')$, $y'\in \R\mapsto h(y,y')$ and $y\in \R\mapsto h(y,y')$ being nondecreasing, we obtain
\begin{align}
f(t,y,z,&\E[Y^n_t])+{(n+1)(y-h(Y^n_t, \E[Y^n_t]))^- }_{|_{(y,z)=(\tilde Y_t,\tilde Z_t)}}\nn\\&
=f(t,\tilde Y_t,\tilde Z_t,\E[Y^n_t])\le f(t,\tilde Y_t,\tilde Z_t,\E[\bar Y_t])\le \Phi(s,\tilde  Y_{s},\E[\bar Y_s])\le \Phi(s,\tilde  Y_{s},\E[\bar Y_s])+d\tilde K_t. \nn
\end{align}
Therefore, by the comparison theorem of the solutions of standard BSDEs,  we obtain $Y^{n+1}\le \tilde Y=\bar Y$, which completes the proof.

\ms
\nd \underline{Step 2}. For any $t\le T$, we set
$Y_t=\lim_{n\rightarrow \infty}Y^n_t$. The process $Y$ is rcll and
there exist a $\mathcal{P}$-measurable process $Z$ such that for any
$p\in [1,2)$, $(Z^n)_{n\ge 0}$ converge to $Z$ in $L^p(dt\otimes
d\P)$. 

\ms
\nd First, note that the process $Y$ exists since by Step 1 the
sequence of processes $(Y^n)_{n\ge 0}$ is non-decreasing and $Y^n\le
\bar Y$. On the other hand since $Y^0$ and $\bar Y$ belongs to
$\Sc^{2}_c$ then 
\begin{equation}\lb{bornyn}
\E[\sup_{t\le T}|Y_t|^2]\le C.
\end{equation}
Furthermore, using It\^o's formula and taking into account the
representation of $Y^n$ similar to \eqref{snellenvelopeyn} we have, for any $t\le T$ and $\eps>0$, 
\begin{align} \nn
(Y^n_{t})^2+\int_{t}^{T}|Z^n_s|^2ds&=\xi^2+2\int_{t}^{T}Y^n_sf(s,Y^n_{s},Z^n_{s},\E[Y^{n-1}_s])
ds+2\int_{t}^{T}Y^n_sdK^n_s-2\int_{t}^{T}Y^n_sZ^n_{s}\,dB_{s}\\
\nn &\le
\xi^2+2\int_{t}^{T}Y^n_s\{f(s,0,0,0)+a(s)Y^n_{s}+b(s)Z^n_{s}+c(s)\E[Y^{n-1}_s])
ds\\&\qquad\qquad +2K^n_T\sup_{s\le T}\{Y^n_s\wedge
h(Y^{n-1}_s,\E[Y^{n-1}_s])\} -2\int_{t}^{T}Y^n_sZ^n_{s}\,dB_{s}\\
\nn &\le
\xi^2+2\int_{t}^{T}Y^n_s\{f(s,0,0,0)+a(s)Y^n_{s}+b(s)Z^n_{s}+c(s)\E[Y^{n-1}_s])
ds\\&\qquad\qquad +\eps(K^n_T)^2+\eps^{-1}(\sup_{s\le
T}\{Y^n_s\wedge
h(Y^{n-1}_s,\E[Y^{n-1}_s])\})^2-2\int_{t}^{T}Y^n_sZ^n_{s}\,dB_{s}
\end{align}
where $a(\cdot)$, $b(\cdot)$ and $c(\cdot)$ are measurable processes bounded by the Lipschitz constant of $f$
$C_f$. Now, taking into account of the
inequalities of Step 1 and the monotonicity property of $h$ to
deduce that, for any $n\ge 1$,  $$ \E[(\sup_{s\le T}\{Y^n_s\wedge
h(Y^{n-1}_s,\E[Y^{n-1}_s])\})^2]\le \E[(\sup_{s\le T}h(\bar
Y_s,\E[\bar Y_s])^2]+\E[(\sup_{s\le T}\{Y^1_s\wedge
h(Y^{0}_s,\E[Y^{0}_s])\})^2].$$ As $$ K^n_T=
Y^n_{0}-\xi-\int_{0}^{T}f(s,Y^n_{s},Z^n_{s},\E[Y^{n-1}_s])\,
ds+\int_{0}^{T}Z^n_{s}\,dB_{s},\,\,t\le T,
$$
then, by using \eqref{estimgeneral}, the inequality $(u+v+w+z)^2\leq
4(u^2+v^2+w^2+z^2)$, the Burkholder-Davis-Gundy inequality and finally by choosing $\eps$ appropriately we obtain
$$
\E[\int_{0}^{T}|Z^n_s|^2ds]\le C \mbox{ and }
\E[\int_{0}^{T}|f(s,Y^n_{s},Z^n_{s},\E[Y^{n-1}_s])|^2 ds ]\le C
$$
where $C$ is a constant independent of $n$. Consequently, in view of  Theorem 2.1 in \cite{pengmono}, we obtain that
\begin{itemize}
\item [(a)] $Y$ is rcll;
\item[(b)] there exist a $\mathcal{P}$-measurable process $Z$ such that for
any $p\in [1,2)$,
$$
\lim_{n\rightarrow \infty}\E[\int_0^T|Z^n_s-Z_s|^pds]=0.
$$
\end{itemize}

\nd \underline{Step 3}. The process $Y$ is continuous and satisfies
\begin{equation}\lb{snellenvyz}
Y_{t}=\underset{\tau \ge
t}{\esssup}\E[\int_t^{\tau}f(s,Y_s,Z_s,\E[Y_s])ds+
h(Y_\tau,\E[Y_t]_{|t=\tau})1_{\{\tau<T\}}+\xi
1_{\{\tau=T\}}|\Fc_{t}].
\end{equation}
For any $n\ge 1$, in view of  the representation \eqref{snellenvelopeyn}, 
we have, for all $t\in [0,T]$,
\begin{align}\lb{snenvyzn}
Y^n_{t}&=\underset{\tau \ge
t}{\esssup}\E[\int_t^{\tau}f(s,Y^{n}_s,Z^n_s,\E[Y^{n-1}_s])ds+Y^n_\tau\wedge
h(Y^{n-1}_\tau,\E[Y^{n-1}_t]_{|t=\tau})1_{\{\tau<T\}}+\xi
1_{\{\tau=T\}}|\Fc_{t}]\\ \nn &=G^n(t)+ H^n(t) \nn
\end{align}
where, for all $t\le T$,
\begin{align}
\nn G^n(t)=\underset{\tau \ge
t}{\esssup}\E[\int_t^{\tau}f(s,Y^{n}_s,Z^n_s,\E[Y^{n-1}_s])ds+Y^n_\tau\wedge
h(Y^{n-1}_\tau,\E[Y^{n-1}_t]_{|t=\tau})1_{\{\tau<T\}}+\xi
1_{\{\tau=T\}}|\Fc_{t}]\\\nn\qquad \qquad -\underset{\tau \ge
t}{\esssup}\E[\int_t^{\tau}f(s,Y_s,Z_s,\E[Y_s])ds+Y^n_\tau\wedge
h(Y^{n-1}_\tau,\E[Y^{n-1}_t]_{|t=\tau})1_{\{\tau<T\}}+\xi
1_{\{\tau=T\}}|\Fc_{t}]\end{align} and
\begin{align}H^n(t)=\underset{\tau \ge
t}{\esssup}\E[\int_t^{\tau}f(s,Y_s,Z_s,\E[Y_s])ds+Y^n_\tau\wedge
h(Y^{n-1}_\tau,\E[Y^{n-1}_t]_{|t=\tau})1_{\{\tau<T\}}+\xi
1_{\{\tau=T\}}|\Fc_{t}].\nn\end{align} 
But, for any $t\le T$,
$$
|G^n(t)|\le
\E[\int_0^{T}|f(s,Y^{n}_s,Z^n_s,\E[Y^{n-1}_s])-f(s,Y_s,Z_s,\E[Y_s])|ds|\Fc_{t}]
$$
since $|\underset{\tau \ge
t}{\esssup}\,\Sigma^1_{t,\tau}-\underset{\tau \ge
t}{\esssup}\,\Sigma^2_{t,\tau}|\le \underset{\tau \ge
t}{\esssup}\,|\Sigma^1_{t,\tau}-\Sigma^2_{t,\tau}|$. Therefore, by
the convergence results stated in Step 2 for $(Y^n)_n$ and $(Z^n)_n$
and using Doob's inequality, the sequence of
processes $(G^n)_{n\ge 0}$ converges to 0 in $\Sc^{p}_c$ for any
$p\in [1,2)$. On the other  hand, using the same argument as in
\eqref{snellenvelopey}, the sequence $(H^n)_{n\ge 1}$
converges increasingly to the following rcll process $H$ defined, f any $t\le T$,  by
\begin{align}H(t)=\underset{\tau \ge
t}{\esssup}\E[\int_t^{\tau}f(s,Y_s,Z_s,\E[Y_s])ds+Y_\tau\wedge
h(Y_\tau,\E[Y_t]_{|t=\tau})1_{\{\tau<T\}}+\xi
1_{\{\tau=T\}}|\Fc_{t}].\nn\end{align} Thus, for any $t\le T$,
$Y_t=H(t)$. Finally,  the continuity of $Y$ is obtained in the same
way as in Step 3 of the previous subsection, since $(\ga,\gb)$ satisfies
\eqref{gamma}. The proof of the claim is now complete. \ms

\nd \underline{Step 4}. For any $t\le T$, we define the continuous
process $K$ by
$$\begin{array}{c}
  K_t=
  Y_{0}-Y_{t}-\int_{0}^{t}f(s,Y_{s},Z_s,\E[Y_s])ds+\int_{0}^{t}Z_{s}\,dB_{s}.\end{array}
$$
The triple of processes $(Y,Z,K)$ is a solution of the mean-field
reflected BSDE \eqref{Eqz} and it is minimal. \ms

\nd The process $K$ is nothing but the limit of the sequence
$(K^n)_n$ and since for any $n$, $K^n$ is non-decreasing then $K$ is
also non-decreasing. On the other hand, by Dini's Theorem $(Y^n)_n$ converges to
$Y$  in $\Sc^{2}_c$ and $(Z^n)_n$ converges to $Z$ in
$L^p(dt\otimes d\P)$ for any $p\in [1,2)$, then, by
Burkholder-Davis-Gundy inequality and Jensen's inequality, $(K^n)_n$
converges in $\Sc^{p}_c$ ($p\in [1,2)$) to $K$. 

\noindent Finally, let us show that the Skorohod condition
$\int_0^T(Y_s-h(Y_{s},\E[Y_s]))dK_s=0$ holds. Consider a subsequence which we still denote by $\{n\}$ such that 
$$
\P\mbox{-a.s.},\quad \lim_{n\to\infty}\sup_{s\le T}|K^n_s-K_s|\rightarrow 0.
$$
As pointed out previously, for any $n\ge 0$, we have
$$\begin{array}{c}
\int_0^T(Y^n_s-Y^n_s\wedge h(Y^n_{s},\E[Y^n_s]))dK^n_s=0.\end{array}
$$
On the other hand 
$$\begin{array}{lll}
\int_0^T(Y^n_s-Y^n_s\wedge h(Y^n_{s},\E[Y^n_s]))dK^n_s\\
\qquad\qquad\qquad =\int_0^T\{(Y^n_s-Y^n_s\wedge h(Y^n_{s},\E[Y^n_s]))-(Y_s-h(Y_{s},\E[Y_s]))\}dK^n_s\\\qquad\qquad\qquad\quad\qquad+\int_0^T(Y_s-h(Y_{s},\E[Y_s]))dK^n_s=A^n_1+A^n_2.
\end{array}$$
But, 
$$|A^n_1|\le \sup_{s\le T}|(Y^n_s-Y^n_s\wedge h(Y^n_{s},\E[Y^n_s]))-(Y_s-h(Y_{s},\E[Y_s]))|\times K^n_T.$$
Thus, $\P$-a.s. $\lim_{n\to\infty}A^n_1 \rightarrow 0$, since $K^n_T$ is bounded thanks to the uniform convergence of $(K^n)_n$ to $K$, the uniform convergence of $(Y^n)_n$ to $Y$, 
$\P$-a.s. and in $\Sc^{2}_c$, and finally the fact that $h$ is Lipschitz. On the other hand, in view of Helly's Convergence Theorem (see \cite{kolmo}, pp. 370),  we have
$$
\begin{array}{lll}
\P\mbox{-a.s.},\quad \underset{n\to\infty}{\lim}A^n_2=\int_0^T(Y_s-h(Y_{s},\E[Y_s]))dK_s.\end{array}
$$
Therefore, $\int_0^T(Y_s-h(Y_{s},\E[Y_s]))dK_s=0$ which means that the Skorohod condition is satisfied and then $(Y,Z,K)$ is a solution of \eqref{Eqz}. \ms

Lastly, the minimality property follows in a similar manner as in Theorem \ref{Theor} above. 
\end{proof}

\begin{remark}
Since $(Y^n)_n$ converges to $Y$ in $\Sc^{2}_c$, this entails the convergence of $(Z^n)_n$ to $Z$ in 
$\Hc^{2,d}$.
\end{remark}

\section{Appendix: Comparison of solutions of BSDEs of mean-field
type}

For sake of completeness, we recall the following result related to comparison of solutions of
mean-field BSDEs given in (\cite{blp09}, Theorem 3.2) in a more
general setting. We present here the form which we need it throughout this paper.
\begin{Theorem}[\cite{blp09}, Theorem 3.2]\label{compappen} Let us consider the two following two BSDEs of mean-field type: For any $t\le T$,
\begin{equation*}
 \label{Eqcomp1}
    \begin{array}{ll}
    Y_{t}=\xi+\int_{t}^{T}f(s,Y_{s},Z_s,\E[Y_s])\,
ds+ \int_{t}^{T}Z_{s}\,dB_{s}
    \end{array}
  \end{equation*}
 and
 \begin{equation}
 \label{Eqcomp2}
    \begin{array}{ll}
    Y'_{t}=\xi'+\int_{t}^{T}f'(s,Y'_{s},Z'_s,\E[Y'_s])\,
ds+ \int_{t}^{T}Z'_{s}\,dB_{s}.
    \end{array}
 \end{equation}
We assume that the mappings $f$ and $f'$  satisfy (H3)-(i) (a)-(b) and
$\xi,\xi'$ are square integrable. If
\begin{itemize}
\item[(a)] $\xi \le \xi'$,

\item[(b)] $f(t,Y'_t,Z'_t,\E[Y'_t])\leq f'(t,Y'_t,Z'_t,\E[Y'_t])$,

\item [(c)] $f$ is non-decreasing w.r.t. $y'$.

\end{itemize}

\nd Then $Y\le Y'$.
\end{Theorem}

\begin{Remark} In this theorem, we do not need that $f'$ satisfies
(H3)-(i) (a)-(b) but only the existence of the solution $(Y',Z')$ of
\eqref{Eqcomp2}.
\end{Remark}

\end{document}